\documentclass[reqno]{amsart}

\usepackage[text={400pt,660pt},centering]{geometry}

\usepackage{color}
\usepackage{esint,amssymb}
\usepackage{mathtools}
\usepackage[colorlinks=true, pdfstartview=FitV, linkcolor=blue, citecolor=blue, urlcolor=blue,pagebackref=false]{hyperref}
\usepackage{bm}
\usepackage{mathabx}

\usepackage{accents}
\renewcommand{\mathring}{\accentset{\circ}}

\usepackage{xfrac}

\newtheorem{theorem}{Theorem}[section]
\newtheorem{proposition}[theorem]{Proposition}
\newtheorem{lemma}[theorem]{Lemma}
\newtheorem{corollary}[theorem]{Corollary}

\theoremstyle{remark}
\newtheorem{remark}[theorem]{Remark}

\theoremstyle{definition}

\numberwithin{equation}{section}
\numberwithin{equation}{section}

\newcommand{\NN}{\mathbb{N}}
\newcommand{\N}{\mathbb{N}}

\newcommand{\ZZ}{\mathbb{Z}}
\newcommand{\RR}{\mathbb{R}}
\newcommand{\R}{\mathbb{R}}

\newcommand{\TT}{\mathbb{T}}

\newcommand{\data}{\mathrm{data}}

\newcommand{\ahom}{\overline{\a}}

\newcommand{\e}{\varepsilon}

\newcommand{\ep}{\varepsilon}

\newcommand{\la}{\lambda}

\newcommand{\rst}[1]{\ensuremath{{\mathbin\upharpoonright}%
		\raise-.5ex\hbox{$#1$}}}

\newcommand{\bbf}{\mathbf{f}}
\newcommand{\bchi}{\bm{\chi}}
\newcommand{\bba}{\mathbf{\overline{a}}}
\newcommand{\bha}{\mathbf{a}}
\renewcommand{\a}{\mathbf{a}}

\newcommand{\bfg}{\mathbf{g}}
\newcommand{\bfs}{\mathbf{s}}

\newcommand{\nnn}{|||}
\newcommand{\bbe}{\mathbf{e}}

\usepackage[utf8]{inputenc}

\title[Periodic Homogenization of Spectrum of Schr\"{o}dinger Operators]{Asymptotic expansion of the spectrum for periodic Schr\"{o}dinger operators}
\author{Scott Armstrong}
\address{Courant Institute of Mathematical Sciences, New York University}
\email{scotta@cims.nyu.edu}
\author{Raghavendra Venkatraman}
\address{Courant Institute of Mathematical Sciences, New York University}
\email{raghav@cims.nyu.edu}
\date{\today}

\begin{document}
	
	\maketitle
	\begin{abstract}
		We prove an asymptotic expansion for the eigenvalues and eigenfunctions of Schr\"{o}dinger-type operator with a confining potential and with principle part a periodic elliptic operator in divergence form. We compare the spectrum to the homogenized operator and characterize the corrections up to arbitrarily high order.  
	\end{abstract}
	
	\section{Introduction}
	\label{s.intro}
	\subsection{Motivation and informal summary of results}
	In this paper we are interested in asymptotic expansions of eigenvalues and eigenfunctions of the operator 
	\begin{equation}
		\label{e.Schrodinger}
		\mathcal{L}_\e 
		:= 
		- \nabla \cdot \bha\bigl(\tfrac{\cdot}{\e}\bigr) \nabla + W\,, 
	\end{equation}
	where $\bha(\cdot)$ is a~$\ZZ^d$--periodic, uniformly elliptic coefficient field valued in the~$d\times d$ symmetric matrices,~$W$ is a confining potential that is quadratic at infinity, and~$\ep >0$ is a small parameter.

	\smallskip
	
	The classical theory of periodic homogenization asserts that, as $\e \to 0$, the behavior of the elliptic operator $\mathcal{L}_\e$ is well-approximated by the constant-coefficient homogenized operator 
	\begin{equation}
		\label{e.Schrodinger0}
		\mathcal{L}_0 := - \nabla \cdot \bba \nabla  + W\,,
	\end{equation}
	where~$\bba$ is a constant symmetric matrix called the homogenized matrix. Owing to the growth of $W$ at infinity, both the operators $\mathcal{L}_\e$ as well as $\mathcal{L}_0$ have a compact resolvent, and therefore have a discrete collection of eigenvalues which we denote by~$\{ \lambda_{\e,j} \}_{j\in\N}$ in the case of~$\mathcal{L}_\e$ and~$\{ \lambda_{0,j} \}_{j\in\N}$ in the case of~$\mathcal{L}_0$. These sequences are arranged in nondecreasing order, repeated according to multiplicity, and increase to infinity as the index~$j \to \infty$. 
	The classical theory of homogenization implies that~$\lambda_{\ep,j} \to \lambda_{0,j}$ as~$\ep \to 0$ for each fixed~$j$, with convergence of the corresponding eigenspaces in~$L^2(\R^d)$ (see~\cite{Ke1,Ke2}). 
	In this paper we are concerned with obtaining quantitative information concerning this convergence. 
	
	\smallskip
	
	We would ideally like to obtain asymptotic expansions in the parameter~$\ep$, hopefully identifying the next-order terms in the expansion. 
	Moreover, we are interested in estimates which are quantitative in both parameters~$j$ and~$\ep$, to identify precisely how high in the spectrum our expansions are valid for, as a function of~$\ep$. 
	
	\smallskip
	
	Such questions have been previously addressed in the context of Dirichlet and Neumann eigenvalue problems in bounded domains.
	In~\cite{KLS1}, the authors prove the estimate
	\begin{equation}
		\label{e.KLS}
		\bigl| \lambda_{\e,j} - \lambda_{0,j} \bigr|
		\leq 
		C \ep \lambda_{0,j}^{\sfrac32} \,.
	\end{equation}
	for a constant $C$ which does not depend on~$j$. 
	For identifying the next-order terms in the context of boundary value problems, the geometry of the boundary and its interaction with the periodic lattice plays an important role.  
	The works~\cite{SaVo,MoVo} characterize the limit points of $\frac{\lambda_{\e,j} - \lambda_{0,j}}\ep$ as $\e \to 0,$ for a \textit{simple} eigenvalue $\lambda_{0,j}$ of $-\nabla \cdot \bba \nabla$ with homogeneous Dirichlet boundary conditions.\footnote{There is also an unpublished manuscript in the website of Vogelius that deals with the homogeneous Neumann boundary condition case, with similar conclusions.} 
	The authors in~\cite{SaVo} demonstrate numerically, that for a planar domain that has faces with rational directions, the possible subsequential limits of the first order correction to the eigenvalue $\frac{\lambda_{\ep,j} - \lambda_{0,j}}{\e}$ as $\e \to 0$ can, in general, be a continuum; the authors in~\cite{MoVo} provide a representation formula for the possible subsequential limits of {\color{blue} $\frac{\lambda_{\ep,j} - \lambda_{0,j}}{\e}$ as $\e\to0$,}  in terms of subsequential limits of corrector equations with oscillating boundary conditions. 
	Thus, in general polygonal domains, it is not possible to identify the~$O(\ep)$ term in an asymptotic expansion due to the behavior of solutions in boundary layers. 
	
	\smallskip

	In smooth, uniformly convex domains, the results in~\cite{Prange} and~\cite{Zhuge} identify the~$O(\ep)$ term in the expansion of~$\lambda_{\ep,j}$ in terms of the solutions of boundary-layer problem, with an error of~$O(\ep^{\sfrac32})$ in~$d>2$ and~$O(\ep^{\sfrac 54})$ in two dimensions, where the implicit constants depend also on~$j$. These results use quantitative estimates for the boundary layer problem in homogenization proved in~\cite{GVM,AKMP,SZ1,SZ2}. To go further in the analysis and understand the higher-order terms in the expansion, a finer analysis of the boundary layer problem is required, beyond the current state of the art.
	
	\smallskip
	
	Our motivation for considering the Schr\"{o}dinger-type operator~$\mathcal{L}_\ep$ in~\eqref{e.Schrodinger} and posing the eigenvalue problems in the whole space is to circumvent the need to understand boundary layers, and thereby give a more complete asymptotic expansion. The role of the quadratically growing potential~$W$ is to provide localization for the eigenfunctions and compactness of the resolvent.
	
	\smallskip
	
	Given a \emph{simple} eigenvalue~$\lambda_{0,j}$ of the operator~$\mathcal{L}_0$ defined in~\eqref{e.Schrodinger}, with corresponding normalized eigenfunction~$\phi_{0,j}$, 
	we exhibit asymptotic expansions for~$\lambda_{\ep,j}$ and~$\psi_{\ep,j}$ of the form
	\begin{equation*}
		\lambda_{\ep,j} 
		= 
		\lambda_{0,j} 
		+
		\sum_{p=2}^P
		\ep^p \mu_p
		+
		O\bigl(\ep^{P+1}\bigr)
	\end{equation*}
	and
	\begin{equation*}
		\psi_{\ep,j} 
		=
		\phi_{0,j}
		+ \sum_{p=2}^P \e^p \sum_{k=2}^p \sum_{m=0}^{p-k} \nabla^m U_k: \bchi_{p-k,m,k}\Bigl( x, \frac{x}{\e} \Bigr)
		+
		O\bigl( \ep^{P+1} \bigr), 
	\end{equation*}
	where the sequences~$\{ \mu_p \}_{\{p\geq 2\}}\subseteq\R$ and~$\{ U_p \}_{\{p\geq 2\}} \subseteq L^2(\R^d)$ are constructed explicitly and depend on~$j$ but not on~$\ep$, and the functions~$\bchi_{p,m,k}$ are correctors that contain the~$\ep$-scale wiggles in the eigenfunction. The implicit norm in the expansion of~$\psi_{\ep,j}$ is the strong~$H^1(\R^d)$ norm.

	\smallskip 
	
	The expansions are valid for any~$P \in\N$, but we should be more explicit  about the error term~$O(\ep^{P+1})$. The term is actually 
	\begin{equation*}
		C\gamma(\lambda_{0,j})
		\biggl( \frac{\ep \lambda_{0,j}^{\sfrac 32}}{\gamma(\lambda_{0,j})} \biggr)^{P+1}\,, \qquad \mbox{where} \quad 
		P \leq c \log \biggl|\log \biggl( \frac{\ep \lambda^{\sfrac 32}}{\gamma(\lambda_{0,j})} \biggr)\biggr|\,,
	\end{equation*}
	where the constant~$C$ does not depend on~$j$,~$p$ nor~$\ep$, and~$\gamma(\lambda_{0,j})$ denotes the spectral gap between~$\lambda_{0,j}$ and the nearest eigenvalue of~$\mathcal{L}_0$ which is not equal~$\lambda_{0,j}$:
	\begin{equation}
		\label{e.sgdef}
		\gamma(\lambda_{0,j})
		:=
		\min \bigl\{ 
		| \lambda_{0,k} - \lambda_{0,j} | \,:\,
		\lambda_{0,k} \neq \lambda_{0,j}
		\bigr\}
		\,.
	\end{equation}
	The point of restricting~$P$ as above is that, without this restriction, the implicit constant in the~$O(\ep^{P+1})$ term actually grows doubly exponentially in~$P$, which renders the estimate useless. 
	Note that this estimate coincides with~\eqref{e.KLS} in the case~$P=0$. 
	We remark that some dependence on the spectral gap is  necessary and occurs even in perturbation theory in finite dimensions (i.e., matrices). 
	See Theorem~\ref{t.simple.full} below for the precise statement. 
	
		\smallskip

The expansion above is not standard in homogenization. This is a reflection of the fact that, to our knowledge, higher-order expansions have only been used previously in the periodic (or stationary) setting. Here, however, the potential creates some macroscopic dependence of the coefficients, and the higher-order expansion must intertwine this dependence with the small scales. Indeed, at higher-order (unlike at first-order) the large macroscopic scale will interact with the microscopic scale, and this interaction is what is captured by the correctors~$\bchi_{p-k,m,k}$ when the parameter~$k$ is at least~$2$. It is due to the presence of these terms in the expansion (which we believe are necessary) that the parameter~$P$ in the expansion above cannot exceed~$c\log |\log\ep|$.

	\smallskip
	
	It should be remarked that the expansion is valid in the case~$P=1$, when we have 
	\begin{equation*}
		| \lambda_{\ep,j} - \lambda_{0,j} |
		+
		\| \psi_{\ep,j} - \phi_{0,j} \|_{H^1(\RR^d)}
		\leq 
		\frac{C \ep^2 \lambda_{0,j}^3}{\gamma(\lambda_{0,j})}\,,
	\end{equation*}
	where $\psi_{\e,j} \in L^2(\RR^d)$ is the eigenfunction of $\mathcal{L}_\e$ associated with eigenvalue $\lambda_{\e,j},$ normalized such that ~$\int_{\RR^d} \psi_{\e,j}\phi_{0,j}\,dx = 1.$ 
	In particular, the~$O(\ep)$ term vanishes. This is due to the simple fact that the leading-order correction to the homogenized operator, represented by the (symmetric part of the) third-order homogenized tensor, is zero. 
	
	\smallskip
	
	Our methods yield similar expansions in the case of eigenvalues of~$\mathcal{L}_0$ with multiplicity, but these asymptotic expansions become difficult to describe in complete generality. It is necessary to describe the entire perturbed eigenspace at once, and the multiplicities can bifurcate (or not) at any higher-order level, leaving us with many different cases to enumerate. For simplicity, we present a result in Theorem~\ref{t.multiple.full} which gives a complete asymptotic expansion in the case that the eigenspace bifurcates at the level of~$\ep^2$ (we do this by assumption that a particular matrix has distinct eigenvalues).

	\subsection{Statements of the main results.} \label{ss.mainth}
	Throughout the paper, $d\geq 2$ denotes the spatial dimension,~$\theta  \in (1,\infty)$ is the ellipticity ratio, and we fix positive constants~$\Lambda,\Lambda_0 > 0$ and $\Lambda_- \leq \Lambda_+$. We consider a coefficient field $\bha(\cdot):\R^d \to \R^{d\times d}$ satisfying the following properties:
	\begin{equation}
		\label{h.sym}
		\bha_{ij} = \bha_{ji}, \quad \forall i,j \in \{ 1,\ldots ,d \} \,,
	\end{equation}
	\begin{equation}
		\label{h.ell} 
		|\xi|^2 \leqslant \langle \bha(x)\xi, \xi \rangle \leqslant \theta |\xi|^2 , \quad \forall \xi \in \RR^d, \mbox{ a.e. } x \in \RR^d,
	\end{equation}
	and
	\begin{equation}
		\label{h.per}
		\bha(x + z) = \bha (x) \quad  \forall z \in \ZZ^d, \mbox{ a.e. } x \in \RR^d.
	\end{equation}
	We assume the potential~$W \in C^\infty(\RR^d)$ satisfies, for all $x\in \RR^d$,
	\begin{equation}
		\label{e.Whypothesis.temp}        | \nabla^k W(x) | 
		\leq
		k!\Lambda^k (1+|x|^2)^{\frac12 (2-k)}, \quad \forall k \in \NN_0\, ,
	\end{equation}
	and
	\begin{equation}
		\label{e.wgrowth}
		\Lambda_-|x|^2  \leqslant W(x) \leqslant \Lambda_+|x|^2\,.
	\end{equation}
	
	Throughout we denote 
	\begin{equation*}
		\data:= (d,\theta, \Lambda,\Lambda_0,\Lambda_+,\Lambda_-)\,.
	\end{equation*}
	So when we say that a constant~$C$ depends on~$\data$, we mean that it depends on the parameters~$(d,\theta, \Lambda,\Lambda_0,\Lambda_+,\Lambda_-)$.
	
	\smallskip
	
	The hypotheses \eqref{h.sym},\eqref{h.ell} and \eqref{e.wgrowth} imply that, for every $\e > 0,$ {\color{blue} setting $\a_\e(\cdot) := \a(\frac{\cdot}{\e}),$} the Schr\"{o}dinger operator 
	\begin{align*}
		\mathcal{L}_\e := -\nabla \cdot \bha_\e \nabla + W(x)\,,
	\end{align*}
	is a positive operator which has discrete spectrum in $L^2(\RR^d).$ To be precise, the eigenvalues of~$\mathcal{L}_\e$ can be put into a sequence~$\{\lambda_{\e,j}\}_{j=1}^\infty \subseteq (0,\infty)$, with $\lambda_{\e,j} \leqslant \lambda_{\e,j+1}$ for every $j \in \NN_0$ and $\lambda_{\e,j} \to +\infty$ as $j \to \infty.$ These eigenvalues evidently have no finite cluster points and are repeated according to multiplicity, with every eigenvalue having finite multiplicity. 
	Associated to these eigenvalues $\{\lambda_{\e,j}\}_{j=1}^\infty$ are eigenfunctions $\{\phi_{\e,j}\}_{j=1}^\infty \subset H^1(\RR^d)$ that may be assumed to be orthonormal in $L^2(\RR^d). $
	
	\smallskip
	
	We let~$\bba$ denote the homogenized matrix corresponding to~$\a(\cdot)$ in the standard theory of periodic homogenization. It satisfies that same uniform ellipticity estimate~\eqref{h.ell}, and the operator~$\mathcal{L}_0$ defined in~\eqref{e.Schrodinger0} captures the leading order asymptotics of the operators~$\mathcal{L}_\e$ as~$\e \to 0$.  
	We arrange the eigenvalues of~$\mathcal{L}_0$ in a nondecreasing sequence~$\{\lambda_{0,j}\}_{j=1}^\infty \subset (0,\infty)$, with eigenvalues repeated according to (finite) multiplicity, with~$\lambda_{0,j} \to \infty$ as~$j \to \infty.$  Associated to the eigenvalues~$\{\lambda_{0,j}\}_{j=1}^\infty$ are $L^2$-normalized eigenfunctions of~$\mathcal{L}_0,$ denoted by~$\{\phi_{0,j}\}_{j=1}^\infty$.
	For any eigenvalue $\lambda_{0,j}$ of $\mathcal{L}_0, $ we define the spectral gap $\gamma(\lambda_{0,j})$ as in~\eqref{e.sgdef}. 
	
	\smallskip
	
	We organize our results in four theorems: Theorems \ref{t.simple} and \ref{t.multiple} give the \textit{first-order} expansions for simple and multiple eigenvalues, respectively, and their associated eigenfunctions to errors that are $O(\e^2),$ with dependence of the prefactor on the eigenvalue. Theorems \ref{t.simple.full}
	and \ref{t.multiple.full} present the higher-order expansions for the eigenvalues and eigenfunctions associated to simple and multiple eigenvalues respectively. 
	
	\begin{theorem}
		\label{t.simple}
		Fix~$j\in\N$ such that~$\lambda_{0,j}$ is a simple eigenvalue of~$\mathcal{L}_0$. 
		There exist constants~$c(\data) \in (0,1]$ and $C(\data)<\infty$ such that, if~$\ep$ satisfies
		\begin{equation} \label{e.eps.condition}
			0<
			\e
			\leq 
			c \gamma(\lambda_{0,j})  \lambda_{0,j}^{-\sfrac32}
			\,,
		\end{equation}
		then the $j$th eigenvalue $\lambda_{\e,j}$ of $\mathcal{L}_\e$ is simple
		and satisfies the estimate
		\begin{equation}
			\label{e.simple.expansion}
			\bigl| \lambda_{\e,j} - \lambda_{0,j} \bigr|
			\leq 
			\frac{C\e^2 \lambda_{0,j}^3}{\gamma(\lambda_{0,j})}
		\end{equation}
		and, moreover, if we let~$\psi_{\e,j}$ denote the corresponding eigenfunction for~$\lambda_{\e,j}$, normalized according to  
		\begin{equation} 
			\label{e.simple.normalization}
			\int_{\RR^d} \psi_{\e,j} \phi_{0,j}\,dx = 1,
		\end{equation}
		then we have the estimate
		\begin{equation} \label{e.simple.eigf}
			\bigl\| \psi_{\e,j} - \bigl( \phi_{0,j} + \e\nabla \phi_{0,j}\cdot\bchi^{(1)}(\tfrac\cdot \e) 
			\bigr) \bigr\|_{H^1(\RR^d)}
			\leq
			\frac{C\e^2 \lambda_{0,j}^3}{\gamma(\lambda_{0,j})}
			\,.
		\end{equation}
	\end{theorem}
\begin{remark}
			The condition~\eqref{e.eps.condition} is optimal with respect to the homogenization regime. Indeed, the estimate~\eqref{e.KLS} asserts that the eigenvalues of the $\mathcal{L}_\e$ operator do not deviate from those of $\mathcal{L}_0$ by more than $C \e\lambda_{0,j}^{\sfrac32}.$ The condition~\eqref{e.eps.condition} essentially guarantees that so long as this deviation does not exceed the spectral gap of $\mathcal{L}_0$ at the eigenvalue $\lambda_{0,j}$, then the eigenvalues of the homogenized operator approximate those of the heterogeneous operator to quadratic order. 
			\end{remark}
	The previous result, which expands a simple eigenvalue to a precision of roughly~$\ep^2$, is a special case of our next result which provides a higher-order expansion to a  precision of roughly~$\ep^{c \log |\log \ep|}$. This higher-order expansion is given in terms of certain objects---namely the homogenized tensors~$\bba_{q,m,k}$, correctors~$\bchi_{q,m,k}$, and the sequences of corrections to the eigenvalues~$\{ \mu_k\}_{k \geq 2}$ and eigenfunctions~$\{ U_k\}_{k \geq 2}$---which are  defined in Section~\ref{s.simple} via a recursive construction. 
	
	\begin{theorem}
		\label{t.simple.full}
		Under the hypotheses of Theorem~\ref{t.simple}, there exist constants~$c(\data)\in (0,1]$ and $C(\data) > 1$ such that, if $\e>0$ satisfies \eqref{e.eps.condition} and we define 
		\begin{align*}
			P := 
			\Biggl\lfloor c\log \,\biggl| \log \biggl( \!\frac{\e \lambda_{0,j}^{\sfrac32}}{\gamma({\lambda_{0,j}})}\!\biggr)
			\biggr|
			\Biggr\rfloor\,,
		\end{align*}
		along with 
		\begin{align*}
			\left\{
			\begin{aligned}
				&\widetilde{\lambda}_\e := \lambda_{0,j} + \sum_{p=2}^P \e^p \mu_p, \\
				&w_\e := \phi_{0,j} + \sum_{p=2}^P \e^p \sum_{k=2}^p \sum_{m=0}^{p-k} \nabla^m U_k: \bchi_{p-k,m,k}\Bigl( x, \frac{x}{\e} \Bigr),
			\end{aligned}
			\right.
		\end{align*} 
		then the $j$th eigenvalue $\lambda_{\e,j}$ of~$\mathcal{L}_\e$ is simple,  and its associated eigenfunction $\psi_{\e,j}$ of $\mathcal{L}_\e$ normalized according to~\eqref{e.simple.normalization}, 
		admits the asymptotic expansion 
		\begin{equation*}
			\frac{1}{\gamma(\lambda_{0,j})}
			\bigl(
			|\lambda_{\e,j} - \widetilde{\lambda}_\e| +
			\|\psi_{\e,j} -  w_\e \|_{H^1(\RR^d)}
			\bigr) 
			\leqslant
			\rho\biggl( \frac{\e \lambda_0^{\sfrac32}}{\gamma({\lambda_0})}\biggr)\, ,
		\end{equation*}
		where the modulus $\rho:(0,1) \to (0,\infty)$ is defined by 
		\begin{equation*}
			\rho(t) := C t^{c\log |\log t|}\,. 
		\end{equation*}
	\end{theorem}
	
	Concerning multiple eigenvalues, once again we offer two theorems: the analog of Theorem \ref{t.simple} is in Theorem \ref{t.multiple} below where we provide the first order expansions for $N > 1$ eigenvalue-eigenfunction pairs of $\mathcal{L}_\e,$ that coalesce into a single eigenvalue of the homogenized operator $\mathcal{L}_0$ of high multiplicity $N.$ The analog of Theorem \ref{t.simple.full} is Theorem \ref{t.multiple.full}, which contains the high order asymptotic expansion for multiple eigenvalues.
	For these theorems, we make a simplifying assumption that a certain symmetric matrix which arises in the analysis has distinct eigenvalues, which ensures that the eigenspace bifurcates into~$N$ distinct branches. This is the matrix~$\mathbb{D}$ given in~\eqref{e.good.Drs}. We expect that this assumption is generic, although certainly not always satisfied. In the case it is not satisfied, one needs to study another such symmetric matrix which occurs at a higher-order level in the expansion. For a general result, one needs to study all possible splittings of the eigenspace at all possible levels in the analysis, which is something we do not attempt to describe fully here. 
	
	\begin{theorem}
		\label{t.multiple} 
		Fix $j \in \NN$ such that $\la_{0,j}$ is a multiple eigenvalue of $\mathcal{L}_0$ of multiplicity $N \geqslant 1$, labeled such that $\la_{0,j} = \la_{0,j+1} = \ldots = \la_{0,j+N-1}$. Let $\{ \phi_{0,j+r}\}_{r=0,\ldots,N-1}$ be an othonormal basis for the associated eigenspace. 
		Assume that the~$N$-by-$N$ symmetric matrix~$\mathbb{D}$ defined in~\eqref{e.good.Drs} has~$N$ distinct eigenvalues. 
		Then there exist constants $c(\data) \in (0,1]$ and $C(\data) < \infty$ such that, if $\e $ satisfies 
		\begin{align}
			0 < \e \leqslant c \gamma(\la_{0,j})\la_{0,j}^{-\sfrac32} \,,
		\end{align}
		then for each $r = 0,\ldots, N-1$,
		\begin{align*}
			|\la_{\e,j+r} - \la_{0,j+r}| \leqslant \frac{C \e^2\la_{0,j}^3}{\gamma(\la_{0,j})}\,,
		\end{align*}
		and moreover, there exists an orthogonal matrix $E \in \RR^{N \times N}$ with $E = (e^r_s)$, such that if we normalize  the associated eigenfunctions (and relabel them using $\psi_{\e,j+r}, r = 0,\ldots, N-1$) according to 
		\begin{equation}
			\label{e.normalization}
			\int_{\RR^d} \psi_{\e,j+r} \phi_{0,j+s}\,dx = e^r_s\,,
		\end{equation}
		then we have the estimate 
		\begin{align}
			 \Bigl\|\psi_{\e,j+r} - \sum_{s=0}^{N-1} e^r_s \bigl( \nabla \phi_{0,j+r} + \e\nabla \phi_{0,j+r} : \bchi^{(1)}(\tfrac{\cdot}{\e})\bigr)\Bigr\|_{H^1(\RR^d)} 
			\leqslant 
			\frac{C \e^2\la_{0,j}^3}{\gamma(\la_{0,j})}\,.
		\end{align}
	\end{theorem}

	Our final result concerns a higher-order asymptotic expansion for the spectrum of~$\mathcal{L}_\e$ near an eigenvalue of~$\mathcal{L}_0$ with multiplicity.
	
	\begin{theorem}
		\label{t.multiple.full}
		
		Under the hypotheses of Theorem~\ref{t.multiple}, there exist constants~$c(\data)\in (0,1]$ and $C(\data) > 1$ such that, if the matrix $E = (e^r_s)_{r,s=0,\ldots,N-1}$ are as in Theorem \ref{t.multiple},  $\e>0$ satisfies \eqref{e.eps.condition} and we define 
		\begin{equation*}
			P := 
			\Biggl\lfloor c\log \,\biggl| \log \biggl( \!\frac{\e \lambda_{0,j}^{\sfrac32}}{\gamma({\lambda_{0,j}})}\!\biggr)
			\biggr|
			\Biggr\rfloor\,, \quad 
			U_{0,j+r} := \sum_{s=0}^{N-1} e^r_s \phi_{0,j+s}\,,
		\end{equation*}
		along with 
		\begin{align*}
			\left\{
			\begin{aligned}
				&\widetilde{\lambda}_{\e,j+r} := \lambda_{0,j} + \sum_{p=2}^P \e^p \mu_{p,j+r}, \\
				&w_\e := U_{0,j+r} + \sum_{p=2}^P \e^p \sum_{k=2}^p \sum_{m=0}^{p-k} \nabla^m U_{k,j+r}: \bchi_{p-k,m,k,r}\Bigl( x, \frac{x}{\e} \Bigr),
			\end{aligned}
			\right.
		\end{align*} 
		then for each $r = 0,\ldots, N-1 ,$ the eigenvalues $\lambda_{\e,j+r}$ of~$\mathcal{L}_\e$ and their associated eigenfunctions, $\{\psi_{\e,j+r}\}_{r=0}^{N-1}$ of $\mathcal{L}_\e$ normalized according to~\eqref{e.normalization}, 
		admit the asymptotic expansion 
		\begin{equation*}
			\frac{1}{\gamma(\lambda_{0,j})}
			\Bigl(
			|\lambda_{\e,j} - \widetilde{\lambda}_\e| +
			\|\psi_{\e,j} -  w_\e \|_{H^1(\RR^d)}
			\Bigr) 
			\leqslant
			\rho\biggl( \frac{\e \lambda_{0,j}^{\sfrac32}}{\gamma({\lambda_{0,j}})}\biggr)\, ,
		\end{equation*}
		where the modulus $\rho:(0,1) \to (0,\infty)$ is defined by 
		\begin{equation*}
			\rho(t) := C t^{c\log |\log t|}\,. 
		\end{equation*}

	\end{theorem}

While the above results have been stated in the periodic case, the analysis extends to the stochastic setting. In that case, we would need to use the optimal quantitative estimates for correctors (see for instance~\cite{AKM,AK} and the references therein), and it would be necessary to stop the expansion after a finite order~$P$ depending on the dimesion~$d$ and the rate of decorrelations of the random coefficient field (since the correctors do not exist after a certain finite order in the random setting).

	\section{Preliminaries}  \label{s.prelim}

	\subsection{The First and Second Order Homogenized Tensors.} \label{ss.abar} We introduce the first and second order correctors, their associated homogenized tensors $\bba, \bba^{(3)},$ and prove a symmetry property of $\bba^{(3)}$ which will play a crucial role in our analysis. These correctors, as well as the first and third homogenized tensors will arise in our infinite order expansion subsequently, however, for the time being we prefer to set some notation that is less heavy (and is well-known) to experts in homogenization. 
	
	For each~$e\in \RR^d$ we let $\chi^1_{e} \in H^1(\TT^d)$ denote the first-order corrector, that is, the unique mean-zero periodic solution of 
	\begin{equation}
		\label{e.first.ord.corr}
		-\nabla \cdot \bha(e +\nabla \chi^1_{e} ) = 0, \quad \quad \langle \chi^1_{e_k} \rangle = 0\,.
	\end{equation}
	The homogenized tensor~$\bba$ is defined by the formula
	\begin{align*}
		\bba e := \langle \bha (e + \nabla \chi^1_{e}) \rangle \,, \quad e\in\RR^d
	\end{align*}
	We let~$\bfg_{e} $ denote the difference between the flux of the corrector and the homogenized flux: 
	\begin{equation*}
		\bfg_{e} := \bha(e + \nabla \chi^1_{e}) - \bba e\,.
	\end{equation*}
	We introduce an associated stream matrix~$\mathbf{s}_e$, which is skew-symmetric and satisfies
	\begin{equation}
		\label{e.divse}
		\nabla \cdot \mathbf{s}_e = \bfg_e
		\qquad \mbox{(in coordinates, $\partial_{x_i} \bfs_{e,ij} = \bfg_{e,j}$)}\,,
	\end{equation}
	and whose $ij$th entry~$\mathbf{s}_{e,ij}$ is defined as the unique mean-zero $H^1(\TT^d)$ solution of
	\begin{equation}
		\label{e.se.eq}
		-\Delta \bfs_{e,ij} = \partial_{x_j} \bfg_{e,i} - \partial_{x_i} \bfg_{e,j}\, \quad \quad \langle \bfs_{e,ij} \rangle = 0\,.
	\end{equation}
	We call~$\mathbf{s}_{e,ij}$ a \emph{flux corrector}. 
	It is clear from~\eqref{e.se.eq} that~$\bfs_{e}$ is skew-symmetric. 
	To check the condition~\eqref{e.divse}, apply~$\partial_{x_i}$ to both sides of~\eqref{e.se.eq}, sum over~$i$, and use the equation~\eqref{e.first.ord.corr} to obtain, in the sense of distributions, 
	\begin{equation}
		-\Delta ( \nabla \cdot \bfs_{e} )_j =
		-\Delta \bfg_{e,j}. 
	\end{equation}
	Since both~$\bfg_{e,j}$ and~$\nabla \cdot \bfs_{e}$ are of zero mean, it follows that they are equal. 
	
	We also introduce the second-order corrector. Later on we will introduce it as a two-tensor valued mean-zero, periodic field $\bchi^{(2)}:$ this means that it is indexed by two indices, say $\{\chi^2_{e_j \otimes e_k}\}_{j,k=1}^d$ defined to be the unique mean-zero $H^1(\TT^d)$ solution to 
	\begin{equation*}
		-\nabla \cdot \a \nabla \chi^{2}_{e_j\otimes e_k}
		=
		\nabla \cdot \bigl( \a e_j \chi^1_{e_k} - \mathbf{s}^1_{e_k} e_j \bigr)\,.
	\end{equation*}
	Introducing the third order homogenized tensor
	\begin{align*}
		\bba^{(3)}_{ijk} := (\a \nabla \chi^2_{e_j\otimes e_k} + \a e_j \chi^1_{e_k} - \bfs^1_{e_k}e_j)_i\,,
	\end{align*}
	the preceding equation asserts that the vector field 
	\begin{align*}
		\a \nabla \chi^2_{e_j \otimes e_k} + \a e_j \chi^1_{e_k} - \bfs^1_{e_k}e_j - \bba^{(3)}_{ijk}e_i
	\end{align*}
	is mean-zero and divergence free; it then follows as before that there exists a skew-symmetric tensor field $\bfs_{e_j \otimes e_k} $ such that 
	\begin{align*}
		(\nabla \cdot \bfs_{e_j\otimes e_k} )_i =  	(\a \nabla \chi^2_{e_j \otimes e_k} + \a e_j \chi^1_{e_k} - \bfs^1_{e_k}e_j )_i - \bba^{(3)}_{ijk}\,.
	\end{align*}
	
	The following lemma collects a fundamental symmetry property of the third-order homogenized tensor $\bba^{(3)}$. 
	
	\begin{lemma}
		\label{l.3rdorder.symm}
		For each $i,j,k \in \{1,\ldots, d\},$ setting 
		\begin{equation*}
			\bba^{(3),s}_{ijk} := \frac{\bba^{(3)}_{ijk} + \bba^{(3)}_{ikj}}{2}\,,
		\end{equation*}
		we have the identity
		\begin{equation}
			\label{e.3symm}
			\bba^{(3),s}_{ijk} + \bba^{(3),s}_{jki} + \bba^{(3),s}_{kij} = 0\,.
		\end{equation}
	\end{lemma}
	\begin{proof}
		Utilizing the equations for the first and second order correctors, along with the skew symmetry of $\bfs^1_{e_k}$, we find 
		\begin{equation*}
			\begin{aligned}
				2\bba^{(3),s}_{ijk}  &=  \langle \a (\nabla \chi^2_{e_j \otimes e_k} + \nabla \chi^2_{e_k \otimes e_j} ) \cdot e_i   \rangle  + \langle \a_{ik}\chi^1_{e_j} + \a_{ij} \chi^1_{e_k} \rangle 		\\
				&= - \langle \nabla \chi^2_{e_j \otimes e_k} + \nabla \chi^2_{e_k \otimes e_j}, \a \nabla \chi^1_{e_i} \rangle + \langle \a_{ik}\chi^1_{e_j} + \a_{ij} \chi^1_{e_k} \rangle\\
				&= \langle -\a_{jm} \partial_{x_m} \chi^1_{e_i} \cdot \chi^1_{e_k} - \a_{km}\partial_{x_m} \chi^1_{e_i} \cdot \chi^1_{e_j} \rangle \\
				&\quad + \langle \a_{jm}\partial_{x_m} \chi^1_{e_k} \chi^1_{e_i} + \a_{km}\partial_{x_m} \chi^1_{e_j}
				\chi^1_{e_i} \rangle 		+ 2 \langle \a_{jk}\chi^1_{e_i} \rangle\,.
			\end{aligned}
		\end{equation*}
		The desired identity follows by cyclically summing over $i,j$ and $k.$ 
	\end{proof}

	\subsection{Zeroth Order Estimates for Eigenvalues} \label{ss.KLS}
	The following proposition, which gives a first estimate for the rate of convergence of the eigenvalues of $\mathcal{L}_\e$ towards those of $\mathcal{L}_0$, is a straightfoward modification of the results in \cite{KLS1}; it forms an ingredient in the proofs of all of our theorems. While the result in that paper lies in the setting of homogenization of the Dirichlet eigenvalues of periodic elliptic operators in bounded domains, the proof, which is based on Courant's minimax characterization of the eigenvalues, adapts readily to our setting.  This basic estimate provides an upper bound for how much eigenvalues of $\mathcal{L}_\e$ can move relative to those of $\mathcal{L}_0.$ We recall that $\{\lambda_{\e,j}\}_{j=1}^\infty$ (resp. $\{\lambda_{0,j}\}_{j=1}^\infty$) denote the eigenvalues of $\mathcal{L}_\e$ (resp. $\mathcal{L}_0$) written in nondecreasing order. For each eigenvalue $\lambda_{0,j}$ of $\mathcal{L}_0$ we recall that $\gamma(\lambda_{0,j})$ denotes the associated the spectral gap of $\lambda_{0,j},$ (see \eqref{e.sgdef} for the definition). The following proposition can be readily proved as in \cite{KLS1}, and we omit the proof. 
	
	\begin{proposition}
		\label{p.eigvalconv}
		There exists~$C_1 > 0$ independent of~$\e $ and~$k$ such that for each~$k \in \NN$ and for every~$\e > 0$ we have 
		\begin{equation}
			\label{e.zeroth}
			|\lambda_{\e,j} - \lambda_{0,j}|\leqslant C_1 \e\lambda_{0,j}^{\sfrac32}\,. 
		\end{equation}
	\end{proposition}

	\section{Estimates for Analyticity} \label{s.analytic}
	\subsection{Elliptic Estimates}
	\label{ss.decay} 
	The goal of this section is to derive exponential decay estimates for eigenfunctions to the operator 
	\begin{align*}
		\mathcal{L}_0U := - \nabla \cdot \bba \nabla U + W(\cdot) U.
	\end{align*}
	Let $\lambda > 0$ be an eigenvalue of $\mathcal{L}_0,$ corresponding to eigenfunction $U,$ with $\|U\|_{L^2} = 1,$ so that 
	\begin{equation}\label{e.evp}
		-\nabla \cdot \bba\nabla U + W U = \lambda U. 
	\end{equation}
	
	\begin{lemma} 
		\label{l.decayests}
		Let $\lambda$ be an eigenvalue of $\mathcal{L}_0,$ and let $f \in C^\infty \cap L^2(\RR^d)$ be such that $\int_{\RR^d} f \phi_\lambda = 0$ for every eigenfunction $\phi_\lambda$ associated to $\lambda.$ There exist $\alpha(\data) > 0,$ $C_1(\data) >0$ and $C(\data) > 0$ such that the following holds: for every solution~$u$ of 
		\begin{align} \label{e.PDE=f.l}
			(\mathcal{L}_0 - \lambda)u = f \quad \mbox{in} \ \RR^d\,,
		\end{align}
		we have the estimate
		\begin{equation} \label{e.cacciopmn}
			\begin{aligned}
				&\int_{\RR^d}(\lambda + |x|^2)^n e^{2H}|\nabla^m u|^2\,dx + \int_{|x| > R_{n,\lambda}} (\lambda + |x|^2)^{n-1} e^{2H} |\nabla^{m-1} u|^2\,dx \\ &\qquad  \lesssim 
				((m-1)!)^2 C_1^{n+m-1}\Lambda^{m-2} \int_{|x| \leqslant R_{n,\lambda}} e^{2H} (\lambda +|x|^2)^{n+m} u^2\,dx \\
				&\qquad \qquad + ((m-1)!)^2 C_1^{m}\Lambda^2 \sum_{\ell = 0}^{m-1} \int_{\RR^d} e^{2H}(\lambda + |x|^2)^{n + m- \ell - 2}|\nabla^\ell f|^2\,dx\,,
			\end{aligned}
		\end{equation}
		where $H: \RR^d \to \RR$ is defined via $H(x) := \alpha |x|^2,$ and 
		\begin{align}
			\label{e.lengthscale-def}
			R_{n,\lambda} := C\max(\sqrt{n},\sqrt{\lambda})
		\end{align}
		
	\end{lemma}
	\begin{proof}
		The proof is an induction argument on $m \in \mathbb{N}.$ 
		
		\smallskip
		
		\emph{The base case.} Let $\eta \in C^\infty_c(\RR^d)$ be a smooth function. We multiply ~\eqref{e.PDE=f.l} with ~$\eta^2 u$ and integrate on ~$\RR^d$ to arrive at 
		\begin{align*}
			\int_{\RR^d} \eta^2\bba \nabla u \cdot \nabla u \,dx + \int_{\RR^d} W(x) \eta^2 u^2\,dx = \lambda \int_{\RR^d}\eta^2 u^2\,dx  + \int_{\RR^d} \eta^2 fu\,dx - 2\int_{\RR^d} \eta u \bba \nabla u \cdot \nabla \eta\,. 
		\end{align*}
		Using ellipticity of $\bba,$ and by Cauchy-Schwarz, we obtain 
		\begin{align} \label{e.cacciop0}
			\int_{\RR^d}\eta^2 |\nabla u|^2 \,dx \lesssim \int_{\RR^d}\eta^2 u^2 (\lambda - W(x) ) + C_0 u^2 |\nabla \eta|^2  \,dx + \int_{\RR^d} \eta^2 f u\,dx, 
		\end{align}
		for some universal constant $C_0$ (for instance we can take $C_0 = 8$). 
		We use test functions of the form $\eta = e^H \chi$ for various choices of  $\chi \in C^\infty_c(\RR^d)$. The function ~$H$ will be chosen of the form ~$H (x):= \alpha |x|^2$ for some small $\alpha = \alpha(\data) > 0.$ Inserting this choice in the prior display we find, for any~$\theta > 0$,
		\begin{align*}
			&\int_{\RR^d}\chi^2 e^{2H}|\nabla u|^2  \,dx +\int_{\RR^d} \chi^2 e^{2H} u^2 \left( W(x) - \frac{C_0}{2}|\nabla H|^2  - \lambda\right) \,dx \\
			&\quad \lesssim    \int_{\RR^d}  e^{2H}u^2(  |\nabla \chi|^2 +  \chi^2 |\nabla H|^2 ) \,dx  +  \int_{\RR^d} e^{2H}\chi^2 \bigl( \frac{f^2}{\theta} +\theta u^2\bigr) \,dx\, .
		\end{align*}
		We thus arrive at the main Caccioppoli estimate that we use for the rest of the proof:  
		\begin{multline}\label{e.cacciop.main}
			\int_{\RR^d}\chi^2 e^{2H}|\nabla u|^2  \,dx +\int_{\RR^d} \chi^2 e^{2H} u^2 \Bigl( W(x) - \frac{C_0}{2}|\nabla H|^2  - (\lambda+\theta)\Bigr)_+ 
			\\
			\lesssim  \int_{\RR^d} e^{2H}\bigl(u^2 |\nabla \chi|^2 + \frac{f^2}{\theta} \chi^2\bigr) \,dx\,.
		\end{multline}
		Here, $(z)_+$ denotes the positive part of $z \in \RR.$ From \eqref{e.wgrowth} we find that 
		\begin{align*}
			W(x) - C|\nabla H|^2 - (\lambda+\theta) \leqslant \Lambda_+ |x|^2 -  \lambda \leqslant 0,
		\end{align*}
		for all $\theta > 0$ if $|x| \leqslant \sqrt{\frac{\lambda}{\Lambda_+}} =: R_0(\lambda).$ In particular, the second term of \eqref{e.cacciop.main} on the left hand side does not contribute on this ball centered at the origin. 
		
		We now make a choice of the $\theta,$ of the test function $\chi,$ and of the exponential weight $H$ toward obtaining the desired estimate.  
		\begin{itemize}
			\item We set $\theta = \frac{\Lambda_-}{8}(\lambda + |x|^2),$
			\item we set $H(x) := \alpha |x|^2$ for $\alpha$ sufficiently small so that $\frac{\Lambda_- }{2} \geqslant 2 C_0 \alpha^2,$ and 
			\item we set $\chi(x) = (\lambda + |x|^2)^{\sfrac{n}{2}} \omega_R(x)$ for $\omega_R \in C^\infty(\RR^d)$ with $\omega_R(x) \equiv 1$ for $|x| \leqslant R, \omega_R(x) \equiv 0$ for $|x| \geqslant 2R$ and $|\nabla \omega_R(x)| \leqslant \frac{1}{R},$ with $|\omega_R(x)| \leqslant 1$ for all $x \in \RR^d.$ Here, $R \gg 1 $ is a parameter that we will send to infinity at the end of the argument of the base case. 
		\end{itemize}
		Then, 
		\begin{align*}
			|\nabla \chi|^2 \lesssim \frac{n^2(\lambda + |x|^2)^{n-1}}{2} + \frac{(\lambda +|x|^2)^n}{R^2}\,. 
		\end{align*}
		Inserting these choices in \eqref{e.cacciop.main} we find 
		\begin{equation} \label{e.basecase.1}
			\begin{aligned}
				&\int_{\RR^d} (\lambda +|x|^2)^n \omega_R^2(x) e^{2H}|\nabla u|^2 \,dx \\
				&\quad + \int_{\RR^d}  (\lambda + |x|^2)^n \omega_R^2(x) e^{2H} u^2 \Bigl( W(x) - \frac{C_0}{2}|\nabla H|^2 - ( \lambda + \frac{\Lambda_-}{8}(\lambda + |x|^2))\Bigr)_+ \,dx\\
				&\quad \leqslant \int_{\RR^d} n^2(\lambda +|x|^2)^{n-1} e^{2H}\Bigl( u^2 + \frac{f^2}{n^2}\Bigr)\,dx + \frac{1}{R^2}\int_{\RR^d} (\lambda+|x|^2)^n e^{2H}u^2\,dx \,. 
			\end{aligned}
		\end{equation}
		Now, if $\tfrac{\Lambda_-}{4}|x|^2 \geqslant \frac{\Lambda_-}{4}\lambda + \Lambda_0 + \lambda$, which we rewrite as $|x| \geqslant R_1(\lambda) \sim \sqrt{\lambda} > R_0(\lambda),$ then choosing $\alpha > 0$ small enough so $\frac{\Lambda_-}{2} \geqslant 2C_0 \alpha^2$, a short calculation shows
		\begin{equation*}
			\begin{aligned}
				W(x) - \frac{C_0}{2}|\nabla H|^2 - \Bigl(\lambda + \frac{\Lambda_-}{8} (\lambda +|x|^2)\Bigr) &\geqslant \Lambda_- |x|^2 - \Lambda_0 - 2 C_0 \alpha^2 |x|^2 - \lambda - \frac{\Lambda_-}{8}(\lambda +|x|^2)\\
				&\geqslant \frac{\Lambda_-}{8}(\lambda +|x|^2) \,. 
			\end{aligned}
		\end{equation*}
		For any $R_{n,\lambda} > R_1(\lambda)$ (a specific choice will be made presently), the estimate ~\eqref{e.basecase.1} then rewrites as 
		\begin{align*}
			&\int_{\RR^d} (\lambda {+}|x|^2)^n \omega_R^2 e^{2H}|\nabla u|^2\,dx 
			+ \int_{|x| \geqslant R_{n,\lambda}} (\lambda {+}|x|^2)^{n-1}\biggl(\frac{\Lambda_-}{8} (\lambda {+}|x|^2)^2 - n^2\biggr) e^{2H} u^2 \omega_R^2(x)\,dx 
			\\ &\quad 
			\lesssim 
			n^2 \int_{|x| \leqslant R_{n,\lambda}} (\lambda {+}|x|^2)^{n-1}e^{2H} u^2 \,dx + \int_{\RR^d} (\lambda {+}|x|^2)^{n-1} e^{2H} f^2\,dx
			\\ & \quad \qquad 
			+ \frac{1}{R^2}\int_{\RR^d} (\lambda{+}|x|^2)^n e^{2H}u^2\,dx \,. 
		\end{align*}
		Choosing $R_{n,\lambda} > R_1(\lambda)$ so that $\frac{\Lambda_-}{16} (\lambda +R_{n,\lambda}^2)^2 \geqslant n^2,$ and then sending $R \to \infty,$ we  obtain
		\begin{align} \label{e.basecasefinal}
			& \int_{\RR^d} (\lambda +|x|^2)^n  e^{2H}|\nabla u|^2\,dx + \frac{\Lambda_-}{16} \int_{|x| \geqslant R_{n,\lambda}} (\lambda + |x|^2)^{n-1} e^{2H} u^2 \,dx \notag \\ & \qquad  
			\lesssim
			n^2 \int_{|x| \leqslant R_{n,\lambda}} (\lambda + |x|^2)^{n-1} e^{2H} u^2 + \int_{\RR^d}(\lambda + |x|^2)^{n-1} e^{2H} f^2\,dx\,.
		\end{align}
		holding for every $n \in \ZZ.$ This completes the base case. We note that $R_{n,\lambda}$ satisfies 
		\begin{align*}
			R_{n,\lambda} > R_1(\lambda) = C\sqrt{\lambda}, \quad \mbox{ and } R_{n,\lambda} > C \sqrt{n}\,.
		\end{align*}
		
		\smallskip
		
		\emph{The induction hypothesis.}
		Suppose that, for each $p \in \{ 1, \cdots, m\},$ we have shown that for every $n \in \NN_0,$ there holds 
		\begin{align}
			\label{e.indhyp}
			\int_{\RR^d} (\lambda {+} |x|^2)^{n} e^{2H}|\nabla^p u|^2 \,dx 
			& 
			\lesssim 
			C_1^{n+p-1} \Lambda^{p-1} ((p{-}1)!)^2\int_{|x| \leqslant R_{n,\lambda}}(\lambda {+} |x|^2)^{n+p}e^{2H}u^2\,dx 
			\notag \\ &
			\quad 
			+ C_1^{p}\Lambda^2 ((p{-}1)!)^2\sum_{j=0}^{p-1} \int_{\RR^d}e^{2H} (\lambda {+} |x|^2)^{n+(p-j)-1}|\nabla^j f|^2\,dx  .
		\end{align}
		
		\smallskip
		
		\emph{The induction step.} Let $\alpha \in \mathbb{N}_0^{m}$ denote a multiindex of length $m.$ As $f \in C^\infty(\RR^d),$ and $\bba$ is a constant matrix, it follows that $u \in C^\infty (\RR^d).$ Setting $v := \partial^\alpha u,$ we apply~$\partial^\alpha$ derivative of~\eqref{e.PDE=f.l} to find that $v$ satisfies the PDE 
		\begin{align}
			\label{e.PDE=fm+1}
			-\nabla \cdot \bba \nabla v + W(x) v = \lambda v +\partial^\alpha f - f_\alpha =: \lambda v + F_\alpha,
		\end{align}
		where, by Leibniz rule, and using \eqref{e.Whypothesis.temp}, $f_\alpha$ satisfies the bound 
		\begin{align*}
			|f_\alpha| &\leqslant \sum_{j=0}^{m} \binom{m}{j} |\partial^j W(x)| |\partial^{m-j} u|
			\leqslant \sum_{j=0}^m \binom{m}{j} j! \Lambda^j (1+|x|^2)^{\frac{1}{2}(2-j)} |\nabla^{m-j}u|\,,
		\end{align*}
		where we used \eqref{e.Whypothesis.temp} in the second inequality above. Applying the base case to \eqref{e.PDE=fm+1} we obtain 
		\begin{equation} \label{e.induct.step}
			\begin{aligned}
				& \int_{\RR^d}(\lambda +|x|^2)^n e^{2H}|\nabla^{m+1}u|^2 \,dx + \int_{|x| \geqslant R_{n,\lambda}} (\lambda +|x|^2)^{n-1} e^{2H} |\nabla^m u|^2\,dx \\  &\quad \lesssim C_1^n \int_{|x| \leqslant R_{n,\lambda}} (\lambda+|x|^2)^{n+1}e^{2H}|\nabla^{m}u|^2\,dx  + \int_{\RR^d}(\lambda +|x|^2)^{n-1}e^{2H}|F_\alpha|^2\,dx .
			\end{aligned}
		\end{equation}
		To complete the induction step, we must estimate the second term. By Cauchy-Schwarz, since $(a_0 + \ldots + a_m)^2 \leqslant (m+1)(a_0^2 + \ldots + a_m^2)$ for any $(m+1)$ positive numbers $a_0,\ldots , a_m$, we find by the induction hypothesis \eqref{e.indhyp}, 
		\begin{align}
			\label{e.Falpha}
			\lefteqn{
				\int_{\RR^d}(\lambda +|x|^2)^{n-1} e^{2H}|F_\alpha|^2\,dx
			} \quad & 
			\notag \\ & 
			\lesssim  \int_{\RR^d}(\lambda +|x|^2)^{n-1} e^{2H}(|\nabla^m f|^2 + |f_\alpha|^2)\,dx  
			\notag \\ & 
			\leq  \int_{\RR^d} e^{2H} (\lambda+|x|^2)^{n-1}|\nabla^m f|^2\,dx 
			\notag \\ & \quad
			+ (m+1)  \sum_{j=0}^m\binom{m}{j}^2(j!)^2\Lambda^{2j} \int_{\RR^d}  (\lambda + |x|^2)^{n-1} (1+|x|^2)^{2-j} e^{2H}|\nabla^{m-j}u|^2\,dx
			\notag \\ & 
			\leq  \int_{\RR^d}e^{2H}(\lambda +|x|^2)^{n-1} e^{2H}|\nabla^m f|^2\,dx    
			\notag \\ & \quad 
			+ (m+1) \biggl(  C_1^{n+m-1}\Lambda^{m-1}((m-1)!)^2\int_{|x| \leqslant R_{n,\lambda}}e^{2H} (\lambda+|x|^2)^{n+m+1} u^2 \,dx 
			\notag \\ & \quad + C_1^m \Lambda^2((m-1)!)^2 \sum_{\ell=0}^{m-1} \int_{\RR^d} e^{2H} (\lambda+|x|^2)^{n-1+(m-\ell)-1}|\nabla^\ell f|^2\,dx \biggr)  
			\notag \\ & \quad
			+ \frac{(m-1)^2m^2}{4} \biggl( \Lambda C_1^{n+m-1} \Lambda^{m-2} (m-2)!^2 \int_{|x| \leqslant R_{n,\lambda}} e^{2H}(\lambda +|x|^2)^{n+m-1}u^2\,dx 
			\notag \\ & \quad 
			+ C_1^{m-1} \Lambda^2(m-2)!^2 \sum_{\ell = 0}^{m-2} \int_{\RR^d} e^{2H} (\lambda +|x|^2)^{n+(m-1-\ell) - 1}|\nabla^\ell f|^2\,dx  \biggr) 
			\notag \\ & \quad
			+ \sum_{j=2}^m \binom{m}{j}^2  \! j!^2 \Lambda^j  \biggl( C_1^{n-1+m-j}\Lambda^{m-j-1} ((m{-}j{-}1)!)^2 \int_{|x| \leqslant R_{n,\lambda}} \!\! (\lambda{+}|x|^2)^{n-1+m-j} e^{2H} u^2\,dx 
			\notag \\ & \quad
			+ C_1^{m-j} \Lambda^2 ((m{-}j{-}1)!)^2\sum_{\ell = 0}^{m-j-1} \int_{\RR^d} e^{2H} (\lambda{+}|x|^2)^{n-1+(m-j-\ell) - 1} |\nabla^\ell f|^2\,dx\biggr)
			\notag \\ &
			\leq  
			m!^2 C_1^{n+m}\Lambda^{m-1} \int_{|x| \leq  R_{n,\lambda}} e^{2H} (\lambda {+}|x|^2)^{n+m+1} u^2\,dx 
			\notag \\ & \quad
			+ m!^2 C_1^{m+1}\Lambda^2 \sum_{\ell = 0}^m \int_{\RR^d} e^{2H}(\lambda {+} |x|^2)^{n + m - \ell - 1}|\nabla^\ell f|^2\,dx. 
		\end{align}
		In the last line, we repeatedly used the fact that $\sum_{j=1}^q t^j \sim t^{q+1}$ with the choice $t = (\lambda + |x|^2)$ to combine the various terms. The proof of the induction step, and therefore that of the lemma, is complete.
	\end{proof}
	Our next lemma concerns $L^2$ estimates for eigenfunctions, i.e., equations of \eqref{e.PDE=f.l} with $f = 0$ with the weight $e^{2H}.$ To this end, we let $\phi_\lambda$, as before, denote an $L^2-$ normalized eigenfunction of $\mathcal{L}_0,$ so that 
	\begin{equation}
		\label{e.L2eigf}
		(\mathcal{L}_0 - \lambda)\phi_\lambda = 0, \quad \quad \mbox{ in } \RR^d, \, \|\phi_\lambda\|_{L^2(\RR^d)} = 1.
	\end{equation}
	\begin{corollary} \label{l.L2decay.eigf}
		Let $\phi_\lambda$ be an eigenfunction of $\mathcal{L}_0$ with eigenvalue $\lambda,$ normalized as in \eqref{e.L2eigf}. Then there exists $c_2(\data) > 0,$ such that
		\begin{align*}
			\int_{\RR^d \setminus B_{R_{1,\lambda}}} \phi_\lambda^2 e^{2H}\,dx \lesssim e^{2c\lambda}\lambda^2\,.
		\end{align*}
	\end{corollary}
	\begin{proof}
		We set $f \equiv 0$ in \eqref{e.cacciopmn}, along with the choice $n = 1$ and $m = 1.$ This yields 
		\begin{equation*}
			\int_{|x| > R_{1,\lambda}} e^{2H}\phi_\lambda^2 \,dx  \lesssim \int_{|x| < R_{1,\lambda}} e^{2H} (\lambda +|x|^2)^2 \phi_\lambda ^2\,dx \lesssim e^{2c \lambda} \lambda^2 \int_{\RR^d} \phi_\lambda^2 \,dx = e^{2c\lambda}\lambda^2\,.  \qedhere
		\end{equation*}
	\end{proof}

	\subsection{Spectral Estimates}

	Next we turn our attention to estimates on the problem 
	\begin{equation}
		\label{e.probf}
		(\mathcal{L}_0 - \lambda) V = f,
	\end{equation}
	where $\lambda \in \sigma(\mathcal{L}_0)$ is a given eigenvalue with associated normalized eigenfunction $\phi_\lambda$, which we assume simple for the time being, and $f \in C^\infty \cap L^2(\RR^d)$ such that 
	\begin{equation*}
		\int_{\RR^d} f\phi_\lambda \,dx = 0.
	\end{equation*}
	Here, and in what follows, we use $\sigma(\mathcal{L}_0)$ to denote the spectrum of $\mathcal{L}_0.$
	Motivated by the decay rates proven in Lemma \ref{l.decayests}, it is natural to measure the regularity of $f$ using weighted spaces with inverse Gaussian weights. 
	Setting $\Lambda_2 := \sqrt{\Lambda},$ we define, for any  $\lambda \in \sigma(\mathcal{L}_0)$, for any $g \in L^2 \cap C^\infty(\RR^d),$ 
	\begin{align}
		\label{e.weightnorm.2}
		\lefteqn{ 
			\nnn g\nnn_{\lambda,\Theta} 
		} \quad & \notag \\ & 
		:= 
		\sup_{n \in \NN}\sup_{m \in \NN}\frac{1}{\Theta^{n+m}\Lambda_2^m n! (m-1)!}\left(\int_{\RR^d} (\lambda +|x|^2)^{n} |\nabla^m g(x)|^2  \exp\bigl(\alpha|x|^2 {-} 2\alpha R_{n,\lambda}^2\bigr)_+ \,dx\right)^{\!\!\sfrac 12} \!\! . 
	\end{align}
	Here, we recall that the lengthscale $R_{n,\lambda} \sim \sqrt{n}$ is defined in \eqref{e.lengthscale-def}. 
	
	\smallskip
	We next give an analyticity estimate for solutions of~\eqref{e.probf}. 
	Recall that $\gamma(\lambda) > 0$ is the spectral gap, $\gamma(\lambda) := \inf\{|\lambda - \mu|: \mu \in \sigma(\mathcal{L}_0) \setminus \{\lambda\}\}.$
	
	\begin{lemma}
		\label{l.reg}
		Let $\lambda$ be an eigenvalue of $\mathcal{L}_0$ and $f \in C^\infty \cap L^2(\RR^d)$ be such that $\int_{\RR^d} f \phi_\lambda = 0$ for every eigenfunction $\phi_\lambda$ associated to $\lambda.$ 
		Then, for every solution~$u$ of 
		\begin{align*} 
			(\mathcal{L}_0 - \lambda)u = f \quad \mbox{in} \ \RR^d\,,
		\end{align*}
		and for every $\Theta \geqslant C_1$, we have the estimate
		\begin{equation*}
			\nnn u\nnn_{\lambda,\Theta} \leqslant \frac{1}{\gamma(\lambda)} \nnn f\nnn_{\lambda,\Theta} \,.
		\end{equation*}
	\end{lemma}
	\begin{proof} The proof proceeds by combining spectral information with Lemma \ref{l.decayests}.
		
		\smallskip
		
		\emph{Step 1.} In this step we obtain the spectral solution formula for $u.$ 
		As $\int f\phi_\lambda = 0$ for each eigenfunction $\phi_\lambda$ associated with the eigenvalue $\lambda,$ it follows that 
		\begin{equation*}
			(\mathcal{L}_0 - \lambda) u = f
		\end{equation*}
		admits a unique solution satisfying the normalization condition $\int u\phi_\lambda = 0$ for each eigenfunction $\phi_\lambda$ associated with eigenvalue $\lambda \in \sigma(\mathcal{L}_0).$ It follows therefore, that if 
		\begin{equation*}
			f := \sum_{\mu \in \sigma(\mathcal{L}_0) \setminus \{\lambda\}} f_\mu \phi_\mu, \quad \quad f_\mu := \int f \phi_\mu. 
		\end{equation*}
		then 
		\begin{equation*}
			u = \sum_{\mu \in \sigma(\mathcal{L}_0) \setminus \{\lambda\}} \frac{f_\mu}{\mu-\lambda}\phi_\mu. 
		\end{equation*}

		\emph{Step 2.} By Lemma \ref{l.decayests}, for any $m \in \NN,$ we obtain that for any $\Theta \geqslant C_1$,
		\begin{align*}
			& \frac{1}{\Theta^{n+m} \Lambda_2^{m} (m-1)!n! }\biggl( \int_{\RR^d} (\lambda +|x|^2)^n  |\nabla^m u|^2 \exp(\alpha |x|^2 - 2\alpha R_{n,\lambda}^2)_+\,dx\biggr)^{\sfrac 12}\\
			&=\frac{1}{\Theta^{n+m} \Lambda_2^{m} (m-1)!n! }
			\\ & \quad \times \biggl[ \int_{|x| \leqslant R_{n,\lambda}} (\lambda +|x|^2)^n |\nabla^m u|^2\,dx + e^{-2\alpha R_{n,\lambda}^2}\int_{|x| \geqslant R_{n,\lambda}}(\lambda +|x|^2)^n  e^{2H} |\nabla^m u|^2\,dx\biggr]^{\sfrac 12} \\
			& \leqslant \frac{1}{\Theta^{n+m} \Lambda_2^{m} (m-1)!n! }\biggl[ \int_{|x| \leqslant R_{n,\lambda}} (\lambda +|x|^2)^n |\nabla^m u|^2\,dx \\ &\quad \quad \quad + e^{-2\alpha R_{n,\lambda}^2}(m-1)!^2 C_1^{n+m-1}\Lambda^{m-2}\int_{|x| \leqslant R_{n,\lambda}}(\lambda +|x|^2)^{n+m}  e^{2H}  u^2\,dx \\&\quad \quad \quad + (m-1)!^2 C_1^m \Lambda^2 e^{-2\alpha R_{n,\lambda}^2} \sum_{\ell=0}^{m-1} \int_{\RR^d} e^{2H} (\lambda+|x|^2)^{n+m-\ell-2}|\nabla^\ell f|^2\,dx \biggr]^{\sfrac 12} \\
			& \leqslant  \biggl[\int_{\RR^d}   u^2\,dx \biggr]^{\sfrac 12} \\ &\quad \quad + \sum_{j=0}^{m-1} \frac{C_1^m \Lambda^2}{\Theta^{n+m} \Lambda^m (m-1)! n!}\biggl[ \int_{\RR^d}(\lambda +|x|^2)^{n +m-j-1} e^{2H}|\nabla^j f|^2\,dx \biggr]^{\sfrac 12} \\
			&\lesssim \biggl[\int_{\RR^d} u^2\,dx\biggr]^{\sfrac 12} + \nnn f\nnn_{\lambda,\Theta} \,.
		\end{align*}
		Taking the supremum with respect to $m,n \in \NN$  on both sides, we arrive at 
		\begin{equation}
			\label{e.triplenormest}
			\nnn u\nnn_{\lambda,\Theta} \lesssim  \biggl[ \int_{\RR^d} u^2\,dx\biggr]^{\sfrac 12} + C \nnn f\nnn_{\lambda,\Theta}.
		\end{equation}
		It remains to estimate the $L^2$ norm of $u.$ Toward this end, using the formula from Step 1, and Plancherel, 
		\begin{align*}
			\int_{\RR^d} u^2 \,dx = \sum_{\mu \in \sigma(\mathcal{L}_0) \setminus \{\lambda\}} \frac{|f_\mu|^2}{(\mu - \lambda)^2} \leqslant \frac{1}{\gamma(\lambda)^2} \sum_{\mu \in \sigma(\mathcal{L}_0) \setminus \{\lambda\}} |f_\mu|^2 = \frac{1}{\gamma(\lambda)^2} \int_{\RR^d} f^2 \,dx,
		\end{align*}
		since $\int f\phi_\lambda = 0.$
		The proof is concluded by observing that $\|f\|_{L^2} \leqslant \nnn f\nnn_{\lambda,\Theta},$ and combining this with \eqref{e.triplenormest}.  
	\end{proof}

	\section{Expansions for Simple Eigenvalues and their Eigenfunctions}
	\label{s.simple}

	In this section we consider the simpler case of a simple eigenvalue of the second-order homogenized operator~$\mathcal{L}_0$ and build an expansion for the corresponding eigenvalue and eigenfunction for the heterogeneous operator~$\mathcal{L}_\e$. Of course, this is only a very particular case of our main results, but the computations and the notation are much less heavy and therefore easier to understand in this setting. 
	
	\subsection{First Order Expansions for Simple Eigenvalues} \label{ss.simple.first}

	\begin{proof}[Proof of Theorem \ref{t.simple}] 
		%
		Define 
		\begin{align*}
			\widetilde{\lambda}_\e := \lambda_{0,j} + \e\mu_{1,j},
		\end{align*}
		where $\mu_{1,j}$ is given by
		\begin{align*}
			\mu_{1,j} := \int_{\RR^d} \bigl(\bba^{(3)}: \nabla^2 \phi_{0,j}(x) \bigr)\cdot \nabla \phi_{0,j}(x)\,dx\,;
		\end{align*}
		we also set 
		\begin{align*}
			w_\e :=  \phi_{0,j} + \e U_{1,j} + \e\nabla (\phi_{0,j} + \e U_{1,j}) \cdot \bchi^{(1)}\bigl(\frac{x}{\e}\bigr)\,,
		\end{align*}
		where we recall that $U_{1,j}$ is the unique solution to the equation 
		\begin{align*}
			(\mathcal{L}_0 - \lambda_{0,j})U_{1,j} = \mu_{1,j} \phi_{0,j} + \bba^{(3)} : \nabla^3 \phi_{0,j},
		\end{align*}
		which is orthogonal in $L^2$ to $\phi_{0,j}.$ In step 1 below we insert this ansatz above in the PDE and compute, and to alleviate notation, suppress the dependence on $j$ (the index of the eigenvalue), which is fixed. 
		
		\smallskip
		\emph{Step 1.} By a direct computation (see, for instance~\cite[Lemma 6.7]{AKM}) , we find that 
		\begin{equation*}
			- \nabla \cdot \bha^\e \nabla w_\e
			=
			- \nabla \cdot \bba \nabla (\phi_{0} + \e U_{1}) 
			+
			\nabla \cdot \biggl(
			\sum_{k=1}^d \bigl( \bfs_{e_k}^\e -  \chi^{1,\e}_{e_k} \bha^\e 
			\bigr)
			\nabla \partial_{x_k} \bigl( \phi_{0} + \e U_{1}\bigr)
			\biggr)
			\,.
		\end{equation*}
		Using the equation for the second-order corrector equation, we can write the second term as
		\begin{align*}
			\lefteqn{
				\nabla \cdot \biggl(
				\sum_{k=1}^d \bigl( \bfs_{e_k}^\e -  \chi^{1,\e}_{e_k} \bha^\e 
				\bigr)
				\nabla \partial_{x_k} \bigl( \phi_{0} + \e U_{1}\bigr)
				\biggr)
			} \quad & 
			\notag \\ & 
			=
			\nabla \cdot 
			\sum_{k=1}^d \bha^\e \chi^{2,\e}_{e_k}
			\nabla \partial_{x_k} \bigl( \phi_{0} + \e U_{1}\bigr)
			-
			\sum_{k=1}^d
			\bigl( \bha^\e \chi^{2,\e}_{e_k}
			+ \bfs_{e_k}^\e -  \chi^{1,\e}_{e_k} \bha^\e 
			\bigr) \nabla^2 \partial_{x_k} \bigl( \phi_{0} + \e U_{1}\bigr)\,.
		\end{align*}
		The first term on the right side is~$O(\e^2)$. The second term on the right side can be rewritten in terms of the second-order flux corrector as follows:
		\begin{align*}
			\lefteqn{
				\sum_{k=1}^d
				\bigl( \bha^\e \chi^{2,\e}_{e_k}
				+ \bfs_{e_k}^\e -  \chi^{1,\e}_{e_k} \bha^\e 
				\bigr) \nabla^2 \partial_{x_k} \bigl( \phi_{0} + \e U_{1}\bigr)
			} \quad & 
			\notag \\ & 
			=
			\e\, \bba^{(3)}:\nabla^3 (\phi_{0}+\e U_{1})
			+
			\sum_{k=1}^d
			\bigl( \bha^\e \chi^{2,\e}_{e_k}
			+ \bfs_{e_k}^\e -  \chi^{1,\e}_{e_k} \bha^\e 
			- \e \,\bba^{(3)}
			\bigr) \nabla^2 \partial_{x_k} \bigl( \phi_{0} + \e U_{1}\bigr)
			\notag \\ & 
			=
			\e\,\bba^{(3)}:\nabla^3 (\phi_{0}+\e U_{1})
			+
			\nabla \cdot
			\sum_{k=1}^d
			\bfs_{e_k}^{2,\e} \nabla^2 \partial_{x_k} \bigl( \phi_{0} + \e U_{1}\bigr)
			\,.
		\end{align*}
		Therefore we obtain that 
		\begin{align}  
			\label{e.wepsdivgrad}
			- \nabla \cdot \bha^\e \nabla w_\e
			& 
			=
			- \nabla \cdot \bba \nabla (\phi_{0} + \e U_{1}) 
			- \e\,
			\bba^{(3)}:\nabla^3 (\phi_{0}+\e U_{1})
			\notag \\ & \qquad 
			+
			\nabla \cdot
			\underbrace{
				\sum_{k=1}^d
				\bigl( \bha^\e \chi_{e_k}^{2,\e} \nabla \partial_{x_k} -
				\bfs_{e_k}^{2,\e} \nabla^2 \partial_{x_k} \bigr) \bigl( \phi_{0} + \e U_{1}\bigr)
			}_{=: R_\e}
		\end{align}
		We note that 
		\begin{align*}
			\|R_\e\|_{L^2} 
			& 
			\leq
			C\e^2 (\|\nabla^2 \phi_0\|_{L^2} + \e \|\nabla^2 U_1\|_{L^2}) + C\e^2 (\|\nabla^3 \phi_0\|_{L^2} + \e \|\nabla^3 U_1\|_{L^2}) 
			\notag \\ & 
			\leq 
			C\e^2 \lambda_0^{\sfrac32} + C\e^3 \frac{\lambda_0^{\sfrac52}}{\gamma({\lambda_0})}\,.
		\end{align*}
		We deduce that 
		\begin{equation} 
			\label{e.Reps.bound}
			\|\nabla \cdot R_\e\|_{H^{-1}(\RR^d)} \leqslant \|R_\e\|_{L^2(\RR^d)} \leqslant C\e^2 \lambda_0^{\sfrac32} + C\e^3 \frac{\lambda_0^{\sfrac52}}{\gamma({\lambda_0})}\,. 
		\end{equation}
		Next, we compute
		\begin{align} \label{e.(W-lam)w}
			&(W(x) - \widetilde{\lambda}_\e) w_\e = (W(x) - \lambda_0)\phi_0 + \e( (W(x) - \lambda_0) U_1 +  \mu_1 U_1) + S_\e,
			\notag \\   &
			S_\e := - \e^2(\mu_1 U_1 + \mu_1\nabla (\phi_0 + \e U_1)\cdot \bchi^{(1)}(\sfrac{x}{\e}) ) + \e \bigl( (W(x) - \lambda_0) \nabla(\phi_0 + \e U_1)\cdot \bchi^{(1)}(\sfrac{x}{\e})\bigr)
			\notag \\
			&=: S_\e^{(1)} + S_\e^{(2)}\,.
		\end{align}
		The term $S_\e^{(1)}$ is clearly of order $O(\e^2)$ in $L^2(\RR^d);$ explicitly, 
		\begin{equation} \label{e.seps1}
			\begin{aligned}
				\|S_\e^{(1)}\|_{L^2(\RR^d)} &\leqslant C\e^2|\mu_1|\bigl( \|U_1\|_{L^2(\RR^d)} + \|\nabla \phi_0 + \e \nabla U_1\|_{L^2(\RR^d)}  \bigr) \\ &\leqslant C\e^2 \lambda_0^{\sfrac{3}{2}} \biggl( \frac{\lambda_0^{\sfrac32}}{\gamma({\lambda_0})} + \lambda_0^{\sfrac12} + \e\frac{\lambda_0^{\sfrac32}}{\gamma({\lambda_0})} \biggr)\,.
			\end{aligned}
		\end{equation}
		Concerning $S_\e^{(2)},$ for each $x \in \RR^d,$ we introduce the function $z \in H^1(\TT^d)$ to be the unique mean zero solution to 
		\begin{align*}
			-\nabla_y \cdot \bha \nabla_y z(x,y)  = (W(x) - \lambda_0) \nabla \phi_0 (x)\cdot \bchi^{(1)}(y)\,.  
		\end{align*}
		This problem is well-posed, since $\langle \bchi^{(1)}(y) \rangle = 0.$ We set $z_\e(x) := \e^3 z(x, \sfrac{x}{\e}),$ and compute that 
		\begin{equation}\label{e.zeps.RHS}
			\begin{aligned} 
				\nabla \cdot \bha^\e \nabla z_\e &= \e^3 \nabla_x \cdot \bha^\e \nabla_x z + \e^2 \bigl( \nabla_y \cdot(\bha^\e \nabla_x z) + \nabla_x \cdot \bha^\e \nabla_y z \bigr)\big\vert_{(x,\sfrac{x}{\e})} + \e \nabla_y \cdot \bha^\e \nabla_y z(x,\sfrac{x}{\e})\\
				&= S_\e^{(3)}(x,\sfrac{x}{\e}) - \e(W(x)- \lambda_0)\nabla \phi_0(x) \cdot \bchi^{(1)}(y) \,, 
			\end{aligned}
		\end{equation}
		
		By Proposition \ref{l.cacciop.per},  for any $\alpha \in \NN_0^d$ and $x \in \RR^d,$
		\begin{align*}
			\|\partial^\alpha_x z(x,\cdot)\|_{H^1(\TT^d)} 
			&\leqslant C\|\partial^\alpha \bigl( (W(x) - \lambda_0 )\nabla \phi_0 (x) \cdot \bchi^{(1)}(y)\bigr) \|_{L^2(\TT^d)} \\ &\leqslant C |\partial^\alpha \bigl( (W(x) - \lambda_0 )\nabla \phi_0 (x)\bigr)|\,.
		\end{align*}
		It follows that 
		\begin{align}  \notag
			\|S_\e^{(3)}(x,\sfrac{x}{\e})\|_{L^2(\RR^d)} &\leqslant C\e^3\|\Delta_x z(x,y)\|_{L^2(\RR^d) \times L^\infty(\TT^d)} + \e^2 \|\nabla_x\nabla_y z(x,y)\|_{L^2(\RR^d) \times L^\infty(\TT^d)} \\ \label{e.seps3bnd}
			&\leqslant C \e^3 \lambda_0^{\sfrac52} + C \frac{\e^2 \lambda_0^2}{\gamma({\lambda_0})}\,.  
		\end{align}
		Combining \eqref{e.wepsdivgrad} through \eqref{e.seps3bnd} we arrive at 
		\begin{equation*}
			\begin{aligned}
				&(\mathcal{L}_\e - \widetilde{\lambda}_\e) (w_\e - z_\e) = \nabla \cdot R_\e + S_\e^{(1)} + \e^2 (W(x) - \lambda_0) \nabla U_1 \cdot \bchi^{(1)}(\sfrac{x}{\e}) + S_\e^{(3)}\\
				&\quad = \nabla \cdot R_\e + \tilde{S}_\e\,.
			\end{aligned}
		\end{equation*}
		Here, 
		\begin{equation}
			\label{e.number}
			\begin{aligned}
	&\|\nabla \cdot R_\e\|_{H^{-1}(\RR^d)} + \|\tilde{S}_\e\|_{L^2(\RR^d)} \\
&\quad \quad \leqslant C\e^2 \lambda_0^{\sfrac32} +  C\e^2 \lambda_0^2  + C\e^3 \lambda_0^{\sfrac52} + \frac{1}{\gamma({\lambda_0})}(C\e^3 \lambda_0^{\sfrac52} + C\e^2 \lambda_0^3 + C\e^3 \lambda_0^3 + C\e^2 \lambda_0^2) =: \delta_\e(\lambda_0)\,.
			\end{aligned}
		\end{equation}

		\smallskip

		\emph{Step 2.} We write ~$R_\e = \sum_{i=1}^\infty c_{\e,i} \phi_{\e,i},$ and $ S_\e = \sum_{i=1}^\infty s_{\e,i} \phi_{\e,i},$ where $\{\phi_{\e,i}\}_{i=1}^\infty$ denote the eigenfunctions of $\mathcal{L}_\e$, normalized so they form an orthonormal basis of $L^2(\RR^d). $ Similarly, we write $w_\e - z_\e = \sum_{i=1}^\infty d_{\e,i} \phi_{\e,i}.$ Then
		\begin{align*}
			(\mathcal{L}_\e - \widetilde{\lambda}_\e )( w_\e - z_\e) =\nabla\cdot R_\e + \tilde{S}_\e,
		\end{align*}
	{\color{blue}	together with~\eqref{e.number}} yields
		
		\begin{align} \label{e.sum.damn.small}
			\sum_{i=1}^\infty |d_{\e,i}|^2(\lambda_{\e,i} - \widetilde{\lambda}_\e)^2 \leqslant \sum_{i=1}^\infty |c_{\e,i}|^2 + |s_{\e,i}|^2\leqslant  \delta_\e(\lambda_0)^2 \,.
		\end{align}
		Now, as $\int_{\RR^d} U_1\phi_0\,dx=0$ by choice,
		\begin{align*}
			\int_{\RR^d}w_\e^2 \,dx = \int_{\RR^d}(\phi_0 + \e U_1 + \e\nabla(\phi_0 +\e U_1)\cdot \bchi^{(1)}(\sfrac{x}{\e}))^2\,dx \leqslant 1 + \delta_\e(\lambda_0)\,.
		\end{align*}
		so that by Plancherel, 
		\begin{align} \label{e.dk.sum.1}
			\sum_{i=1}^\infty|d_{\e,i}|^2 \leqslant  1 + \delta_\e(\lambda_0)\,.
		\end{align}

		\smallskip
		\emph{Step 3.} In this step, we restore notating the dependence of various quantities on $k,$ and obtain an intermediate bound.  
		By Proposition \ref{p.eigvalconv}, for all $\e$ and $j \in \NN,$ we have 
		\begin{align*}
			|\lambda_{\e,i} - \lambda_{0,i}| \leqslant C_1 \e \lambda_{0,i}^{\sfrac32}\,.
		\end{align*}
		As $|\widetilde{\lambda}_\e - \lambda_{0,j}| \leqslant \e|\mu_{1,j}| \leqslant \e C_{2,2,0} \lambda_{0,j}^{\sfrac32}, $ it follows from the triangle inequality that for all $i \neq j$ we have 
		\begin{align*}
			|\lambda_{\e,i} - \widetilde{\lambda}_\e| &= |\lambda_{0,i} - \lambda_{0,j} + \lambda_{0,j} - \widetilde{\lambda}_\e + \lambda_{\e,i} - \lambda_{0,i}|\\
			&\geqslant \gamma(\lambda_{0,j}) - \e (C_1 + C_{2,2,0})\lambda_{0,j}^{\sfrac32}\,.
		\end{align*}
		In particular, if $\e < \frac{\gamma(\lambda_{0,j})}{2(C_{2,2,0} + C_1) \lambda_{0,j}^{\sfrac32}}, $ it follows that 
		\begin{align*}
			|\lambda_{\e,i} - \widetilde{\lambda}_\e| \geqslant \frac{\gamma(\lambda_{0,j})}{2},
		\end{align*}
		for all $i \neq j.$ It follows that 
		\begin{align*}
			\sum_{i \neq j} |d_{\e,i}|^2 \leqslant \frac{\delta_\e(\lambda_{0,j})^2}{\gamma(\lambda_{0,j})^2}\,. 
		\end{align*}
		From \eqref{e.dk.sum.1} it then follows that 
		\begin{align*}
			\bigl| |d_{\e,j}|^2 - 1\bigr| \leqslant \frac{\delta_\e(\lambda_{0,j})^2}{\gamma(\lambda_{0,j})^2}\,. 
		\end{align*}
		Therefore, 
		\begin{align*}
			|\lambda_{\e,j} - \widetilde{\lambda}_\e|^2 \leqslant  \frac{\delta_\e(\la_{0,j})^2}{1 -\frac{\delta_\e(\lambda_{0,j})^2}{\gamma(\lambda_{0,j})^2} } \leqslant  2\delta_\e(\lambda_{0,j})^2\,. 
		\end{align*}
		This implies the estimate desired in \eqref{e.simple.expansion}. Finally, to prove \eqref{e.simple.eigf}, we note that the eigenfunction $\psi_{\e,j}$ associated to the eigenvalue $\lambda_{\e,j},$ which is normalized as in \eqref{e.simple.normalization} satisfies 
		\begin{equation*}
			\begin{aligned}
				(\mathcal{L}_\e - \lambda_{\e,j})(\psi_{\e,j} - (w_\e-z_\e)) &= (\widetilde{\lambda}_\e - \lambda_{\e,j}) (w_\e-z_\e) + \nabla \cdot R_\e + S_\e\,. 
			\end{aligned}
		\end{equation*}
		Then, once again, by Plancherel's theorem we find 
		\begin{equation*}
			\begin{aligned}
				\int_{\RR^d}(\psi_{\e,j} - (w_\e-z_\e))^2 \,dx \lesssim  \sum_{i \neq j} |d_{\e,i}|^2 + \bigl|1 - |d_{\e,j}|^2 \bigr| \lesssim \frac{\delta_\e(\lambda_{0,j})^2}{\gamma(\lambda_{0,j})^2}\,.
			\end{aligned}
		\end{equation*}
		Simiarly, the $H^1$ estimate is proven by differentiating the equation for $\psi_{\e,j} -( w_\e-z_\e),$ and estimating similarly using Plancherel. 
		
		\smallskip
		
		\emph{Step 4.} In order to complete the argument we must show that $\mu_{1,k} = 0$, and so $U_{1,k} \equiv 0.$ Recalling that 
		\begin{align*}
			\mu_{1,k} = \int_{\RR^d} \bba_{ijp}^{(3)} \partial^2_{x_jx_p} \phi_{0,j}(x)
			\partial_{x_i} \phi_{0,j}(x)\,dx \,.
		\end{align*}
		By the symmetry of the hessian this means 
		\begin{align*}
			\mu_{1,k} = \int_{\RR^d} (\bba^{(3)}_{ijp} + \bba^{(3)}_{ipj})\partial^2_{x_jx_p}\phi_{0,j}\partial_{x_i}\phi_{0,j} \,dx = 2 \int_{\RR^d} \bba^{(3),s}_{ijp} \partial^2_{x_jx_p}\phi_{0,j} \partial_{x_i}\phi_{0,j}\,dx\,.
		\end{align*}
		Integrating by parts twice, and then re-indexing we get
		\begin{align*}
			\mu_{1,r} = 2 \int_{\RR^d} \bba^{(3),s}_{ijp} \partial^2_{x_ix_p} \phi_{0,j} \partial_{x_j} \phi_{0,j}\,dx = 2\int_{\RR^d} \bba^{(3),s}_{pij}\partial^2_{x_kx_j}\phi_{0,j} \partial_{x_i}\phi_{0,j}\,dx,
		\end{align*}
		and similarly, 
		\begin{align*}
			\mu_{1,k} = 2\int_{\RR^d} \bba^{(3),s}_{jpi} \partial^2_{x_jx_p}\phi_{0,j}\partial_{x_i}\phi_{0,j}\,dx\,.
		\end{align*}
		Adding, and invoking Lemma \ref{l.3rdorder.symm}, we find that 
		\begin{align*}
			3\mu_{1,k} = 0\,.
		\end{align*}
		It follows then that $U_{1,k}\equiv 0,$ and hence the conclusion of the theorem obtains. 
	\end{proof}

	The rest of this section develops the machinery, and eventually proves the high order expansion for a simple eigenvalue $\lambda_{\e,j}$ of $\mathcal{L}_\e$ from a simple eigenvalue $\lambda_{0,j}$ of $\mathcal{L}_0$ (along with associated expansions for the eigenfunctions). In the remainder of this section, as we work with a fixed simple eigenvalue $\lambda_{0,j}$ of $\mathcal{L}_0$, and we omit the dependence on the index~$j$, henceforth denoting~$\lambda_{0,j}$ by~$\lambda_0$. 
	
	\subsection{Formal expansion and heuristic derivation of the corrector equations} \label{ss.informal.simple}
	
	In this subsection we give some heuristic computations to motivate our asymptotic expansion. 
	Let~$\lambda_0$ be an eigenvalue of $\mathcal{L}_0$ with eigenfunction $u_0,$ and we look for an eigenpair $(u_\e,\lambda_\e)$ of the heterogeneous operator~$\mathcal{L}_\e$ which admits the following ansatz:
	\begin{equation}
		\label{e.ansatzlambdaep}
		\lambda_\e = 
		\lambda_0 + 
		\sum_{j=1}^\infty \e^j \mu_j 
		=
		\sum_{j=0}^\infty \e^j\mu_j\,,
	\end{equation}
	where for convenience we have set~$\mu_0:= \lambda_0$,
	and
	\begin{align}
		\label{e.ansatzuep}
		u_\ep (x)
		&
		= 
		\sum_{k=0}^\infty
		\sum_{m=0}^\infty
		\sum_{|\alpha|=m}
		\sum_{n=m}^\infty
		\ep^{k+n} 
		\partial_x^\alpha U_k(x) 
		\chi_{n,\alpha,k} \biggl( x, \frac{x}{\ep} \biggr)
		\notag \\ & 
		=
		\sum_{p=0}^\infty
		\sum_{k=0}^p
		\sum_{m=0}^{p-k}
		\sum_{|\alpha|=m}
		\ep^p 
		\partial_x^\alpha U_k(x) 
		\chi_{p-k,\alpha,k} \biggl( x, \frac{x}{\ep} \biggr)
		\,.
	\end{align}
	The parameters~$\{ \mu_j \}_{j\in\NN}$ will be determined together with the functions~$\chi_{n,\alpha,k}(x,y)$, called the correctors, which are 
	periodic in the variable~$y$. 
	Note that the correctors~$\chi_{n,\alpha,k}$ are indexed by~$(n,\alpha,k)$, where~$n,k\in\N_0$ and~$\alpha$ is a multiindex. 
	We also denote 
	\begin{equation*}
		\bchi_{n,m,k} = ( \chi_{n,\alpha,k} )_{|\alpha|=m}\,,
	\end{equation*}
	which is a periodic function taking values in~$\mathbf{T}^m$. We may then write the ansatz in the second line of~\eqref{e.ansatzuep} in tensor notation as 
	\begin{align} 
		\label{e.2scexp}
		u_\e(x) 
		& = 
		\sum_{p=0}^\infty \sum_{k=0}^p \sum_{m=0}^{p-k} \e^p  \nabla^m U_k(x) : \bchi_{p-k,m,k}\left(x,\frac{x}{\e}\right)
		\,. 
	\end{align}
	We declare straightaway that the correctors will satisfy the following properties:
	\begin{equation}
		\label{e.chi.00k.meanzero}
		\chi_{0,0,k} = 1, \quad \forall k\in\N,
		\quad \mbox{and} \quad 
		\langle \chi_{p,\alpha,k} \rangle = {\bf{1}}_{\{ p = 0, \alpha = 0 \}}\,.
	\end{equation}
	Since $\chi_{p,\alpha,k}$ only appears in~\eqref{e.ansatzuep} if $p \geq |\alpha|$, and if all the indices are nonnegative, we also adopt the convention that 
	\begin{equation}
		p < |\alpha|  \implies \chi_{p,\alpha,k} = 0, \quad \forall p,\alpha,k;
	\end{equation}
	and that~$\chi_{p,\alpha,k}=0$ if any index is negative. Throughout, if $F(x,y)$ is a function which is periodic in the variable~$y$, then we denote by $\langle F \rangle(x)$ the mean of $F(x,\cdot)$ and set
	\begin{equation}
		\label{e.mathring}
		\mathring{F}(x,y):= F(x,y) - \langle F \rangle(x)\,.
	\end{equation}
	To determine the correctors~$\chi_{q,\alpha,k}$, the parameters $\{ \mu_j\}$ and the macroscopic functions~$\{ U_k \}$, we proceed (informally) by plugging the ansatz for $u^\e$ into the equation $\mathcal{L}_\e u^\e = \lambda_\e u_\e$.
	First, we compute the right hand side of the equation by multiplying~\eqref{e.ansatzlambdaep} and~\eqref{e.ansatzuep}:
	\begin{align}
		\label{e.lamuep}
		\lambda_\ep u_\ep(x)
		&
		=
		\sum_{j=0}^\infty
		\sum_{p=0}^\infty
		\sum_{k=0}^p
		\sum_{m=0}^{p-k}
		\sum_{|\alpha|=m}
		\ep^{p+j} \mu_j
		\partial_x^\alpha U_k(x) 
		\chi_{p-k,\alpha,k} \biggl( x, \frac{x}{\ep} \biggr)
		\notag \\ & 
		=
		\sum_{p=0}^\infty
		\sum_{r=0}^p
		\sum_{k=0}^r
		\sum_{m=0}^{r-k}
		\sum_{|\alpha|=m}
		\ep^{p} \partial_x^\alpha U_k(x) 
		\mu_{p-r}
		\chi_{r-k,\alpha,k} \biggl( x, \frac{x}{\ep} \biggr)
		\notag \\ & 
		=
		\sum_{p=0}^\infty
		\sum_{k=0}^p
		\sum_{m=0}^{p-k}
		\sum_{|\alpha|=m}
		\ep^{p} \partial_x^\alpha U_k(x) 
		\sum_{r=m}^p
		\mu_{p-r}
		\chi_{r-k,\alpha,k} \biggl( x, \frac{x}{\ep} \biggr)
		\,.
	\end{align}
	We turn to the computation of the left side of the equation, namely~$\mathcal{L}_\ep u_\ep$. We first compute the gradient~$\nabla u_\ep$ in coodinates:
	\begin{align*}
		\partial_{x_j} u_\ep(x) 
		&
		=
		\Biggl[
		\bigl( \partial_{x_j} + \ep^{-1} \partial_{y_j} \bigr)
		\sum_{p=0}^\infty
		\sum_{k=0}^p
		\sum_{m=0}^{p-k}
		\sum_{|\alpha|=m}
		\ep^p 
		\partial_x^\alpha U_k(x) 
		\chi_{p-k,\alpha,k} ( x, y)
		\Biggr] \Biggr\vert_{y= \frac{x}{\e}}
		\notag \\ & 
		=
		\sum_{p=0}^\infty
		\sum_{k=0}^p
		\sum_{m=0}^{p-k}
		\sum_{|\alpha|=m}
		\biggl( 
		\ep^{p-1} 
		\partial_x^\alpha U_k(x) 
		\partial_{y_j} \chi_{p-k,\alpha,k} ( x, y)
		+
		\ep^p 
		\partial_x^{\alpha+e_j} U_k(x) 
		\chi_{p-k,\alpha,k} ( x, y)
		\notag \\ & 
		\qquad \qquad \qquad \qquad \qquad 
		+
		\ep^p 
		\partial_x^\alpha U_k(x) 
		\partial_{x_j} \chi_{p-k,\alpha,k} ( x, y)
		\biggr) \biggr\vert_{y= \frac{x}{\e}}
		\notag \\ & 
		=
		\sum_{p=0}^\infty
		\sum_{k=0}^p
		\sum_{m=0}^{p-k}
		\sum_{|\alpha|=m}
		\ep^{p-1} 
		\partial_x^\alpha U_k(x) 
		\notag \\ & 
		\quad
		\times \Bigl( 
		\partial_{y_j} \chi_{p-k,\alpha,k} ( x, y)
		+
		\chi_{p-1-k,\alpha-e_j,k} ( x, y)
		+
		\partial_{x_j} \chi_{p-1-k,\alpha,k} ( x, y)
		\Bigr) \biggr\vert_{y=\frac x\ep}
		\,.
	\end{align*}
	In the last line, we re-indexed two of the sums in order to make the common factor~$\ep^{p-1} \partial_{x}^\alpha U_k$ appear in each of the three terms. This requires changing the bounds on the summands, and the expression we have written actually extra terms in the sum because we did not change the bounds. However, the extra terms correspond to~$\chi_{p,\alpha,k}$ with either~$p=-1$ or~$\alpha=-e_j$. If we adopt the convention that~$\chi_{p,\alpha,k} :=0$ if any index is negative, then the expression above is valid. We will play the same game in our computations below.
	
	\smallskip
	
	We next compute
	\begin{align*}
		\lefteqn{
			(\nabla \cdot \a_\ep \nabla u_\ep)(x)
		} \  & 
		\\ & 
		=
		\sum_{i,j=1}^d
		\bigl( \partial_{x_i} + \ep^{-1} \partial_{y_i} \bigr)
		\sum_{p=0}^\infty
		\sum_{k=0}^p
		\sum_{m=0}^{p-k}
		\sum_{|\alpha|=m} \!\!
		\ep^{p-1} 
		\partial_x^\alpha U_k(x) 
		\notag \\ & 
		\qquad \qquad
		\times  \a_{ij}(y)  \Bigl( 
		\partial_{y_j} \chi_{p-k,\alpha,k} ( x, y)
		+
		\chi_{p-1-k,\alpha-e_j,k} ( x, y)
		+
		\partial_{x_j} \chi_{p-1-k,\alpha,k} ( x, y)
		\Bigr) \biggr\vert_{y=\frac x\ep}
		\notag \\ & 
		=
		\sum_{p=0}^\infty
		\sum_{k=0}^p
		\sum_{m=0}^{p-k}
		\sum_{|\alpha|=m} \!\!
		\ep^{p-2} \partial_x^\alpha U_k
		\notag \\ & 
		\qquad \qquad 
		\times 
		\Bigl(
		\nabla_y \cdot \a \nabla_y \chi_{p-k,\alpha,k} 
		+
		\nabla_y \cdot \a \nabla_x \chi_{p-1-k,\alpha,k} 
		+ 
		\sum_{i,j=1}^d \!\!
		\partial_{y_i} \bigl( \a_{ij} \chi_{p-1-k,\alpha-e_j,k} \bigr) 
		\Bigr)
		\notag \\ & 
		\quad 
		+ 
		\sum_{p=0}^\infty
		\sum_{k=0}^p
		\sum_{m=0}^{p-k}
		\sum_{|\alpha|=m} 
		\sum_{i,j=1}^d
		\ep^{p-1} \partial_x^{\alpha+e_i} U_k
		\a_{ij}  \Bigl( 
		\partial_{y_j} \chi_{p-k,\alpha,k}  
		+
		\chi_{p-1-k,\alpha-e_j,k} 
		+
		\partial_{x_j} \chi_{p-1-k,\alpha,k}  
		\Bigr) 
		\notag \\ & 
		\quad 
		+ 
		\sum_{p=0}^\infty
		\sum_{k=0}^p
		\sum_{m=0}^{p-k}
		\sum_{|\alpha|=m} \!\!
		\ep^{p-1} 
		\partial_x^\alpha U_k 
		\notag \\ & 
		\qquad \qquad 
		\times  
		\Bigl(
		\nabla_x \cdot \a \nabla_y \chi_{p-k,\alpha,k} 
		+
		\nabla_x \cdot \a \nabla_x \chi_{p-1-k,\alpha,k} 
		+ 
		\sum_{i,j=1}^d \!\!
		\partial_{x_i} \bigl( \a_{ij} \chi_{p-1-k,\alpha-e_j,k} \bigr) 
		\Bigr)
		\biggr\vert_{t=\frac x\ep}
		\,.
	\end{align*}
	Changing the bounds on the sums in order to factor out the common term~$\ep^{p-2} \partial_x^\alpha U_k$, we can write this expression as 
	\begin{align*}
		\lefteqn{
			(\nabla \cdot \a_\ep \nabla u_\ep)(x)
		} \ \  & 
		\\ & 
		=
		\sum_{p=0}^\infty
		\sum_{k=0}^p
		\sum_{m=0}^{p-k}
		\sum_{|\alpha|=m} \!\!
		\ep^{p-2} \partial_x^\alpha U_k
		\notag \\ & 
		\qquad  
		\times 
		\biggl(
		\nabla_y \cdot \a \nabla_y \chi_{p-k,\alpha,k} 
		+
		\nabla_y \cdot \a \nabla_x \chi_{p-1-k,\alpha,k} 
		+ 
		\sum_{i,j=1}^d \!\!
		\partial_{y_i} \bigl( \a_{ij} \chi_{p-1-k,\alpha-e_j,k} \bigr) 
		\notag \\ & 
		\qquad  \qquad 
		+
		\sum_{i,j=1}^d
		\a_{ij} 
		\Bigl( 
		\partial_{y_j} \chi_{p-1-k,\alpha-e_i,k}  
		+
		\chi_{p-2-k,\alpha-e_j-e_i,k} 
		+
		\partial_{x_j} \chi_{p-2-k,\alpha-e_i,k}  
		\Bigr) 
		\notag \\ & 
		\qquad  \qquad 
		+
		\nabla_x \cdot \a \nabla_y \chi_{p-1-k,\alpha,k} 
		+
		\nabla_x \cdot \a \nabla_x \chi_{p-2-k,\alpha,k} 
		+ 
		\sum_{i,j=1}^d \!\!
		\partial_{x_i} \bigl( \a_{ij} \chi_{p-2-k,\alpha-e_j,k} \bigr) 
		\biggr)
		\Biggr\vert_{y = \frac x\ep}
		\!\! .
	\end{align*}
	In order to simplify this expression, we introduce the vector field~$\mathbf{f}_{q,\alpha,k}(x,y)$ with $i$th entry given by 
	\begin{equation*}
		\bigl(\mathbf{f}_{q,\alpha,k}\bigr)_i
		:=
		\bigl( \a \nabla_y \chi_{q,\alpha,k} \bigr)_i
		+ 
		\bigl( \a  \nabla_x \chi_{q-1,\alpha,k} \bigr)_i
		+
		\sum_{j=1}^d \a_{ij} \chi_{q-1,\alpha-e_j,k}
		\,.
	\end{equation*}
	Substituting this expression into the previous display, we get
	\begin{align*}
		\lefteqn{
			(\nabla \cdot \a_\ep \nabla u_\ep)(x)
		} \quad & 
		\\ & 
		=
		\sum_{p=0}^\infty
		\sum_{k=0}^p
		\sum_{m=0}^{p-k}
		\sum_{|\alpha|=m} \!\!
		\ep^{p-2} \partial_x^\alpha U_k
		\biggl( 
		\nabla_y \cdot \mathbf{f}_{p-k,\alpha,k}
		+ \nabla_x \cdot  \mathbf{f}_{p-1-k,\alpha,k}
		+
		\sum_{i=1}^d
		\mathbf{f}_{p-1-k,\alpha-e_i,k}
		\biggr) \biggr\vert_{y=\frac x\ep}
		\,.
	\end{align*}
	Combining~\eqref{e.lamuep} with the previous display, remembering also the $W u_\ep$ term, we obtain 
	\begin{align}
		\label{e.expansa}
		\lefteqn{
			\bigl( -\nabla \cdot \a_\ep \nabla + W - \lambda_\ep \bigr) u_\ep 
		} \quad & 
		\notag \\ & 
		=
		\sum_{p=0}^\infty
		\sum_{k=0}^p
		\sum_{m=0}^{p-k}
		\sum_{|\alpha|=m} \!\!
		\ep^{p-2} \partial_x^\alpha U_k 
		\biggl( 
		-\nabla_y \cdot \mathbf{f}_{p-k,\alpha,k}
		- \nabla_x \cdot  \mathbf{f}_{p-1-k,\alpha,k}
		-
		\sum_{i=1}^d
		\mathbf{f}_{p-k,\alpha-e_i,k}
		\notag \\ & 
		\qquad \qquad \qquad \qquad \qquad 
		\qquad \qquad
		+
		W
		\chi_{p-2-k,\alpha,k} 
		-
		\sum_{r=|\alpha|}^{p-2}
		\mu_{p-2-r}
		\chi_{r-k,\alpha,k} 
		\biggr) \biggr\vert_{y=\frac x\ep}
		\!.
	\end{align}
	Our ansatz would obviously be very good if the term inside parenthesis on the right side of~\eqref{e.expansa} was zero (or at least very small), for every $(p,\alpha,k)$. But before aiming for such a lofty goal, we first insist that it be \emph{macroscopic}, that is, independent of the variable~$y$. This is the same as demanding that it be equal to its mean over~$y\in\TT^d$. 
	In view of~\eqref{e.chi.00k.meanzero}, this is 
	\begin{align*}
		& -\nabla_y \cdot \mathbf{f}_{p-k,\alpha,k}
		- \nabla_x \cdot  \mathbf{f}_{p-1-k,\alpha,k}
		-
		\sum_{i=1}^d
		\mathbf{f}_{p-1-k,\alpha-e_i,k}
		+
		W
		\chi_{p-2-k,\alpha,k} 
		-
		\sum_{r=|\alpha|}^{p-2}
		\mu_{p-2-r}
		\chi_{r-k,\alpha,k} 
		\notag \\ & \quad 
		=
		-\nabla_x \cdot \langle \mathbf{f}_{p-1-k,\alpha,k} \rangle
		-
		\sum_{i=1}^d
		\langle
		\mathbf{f}_{p-1-k,\alpha-e_i,k}
		\rangle
		+ W(x) {\bf{1}}_{\{ p=k+2, \alpha = 0 \}}
		-
		\sum_{r=|\alpha|}^{p-2}
		\mu_{p-2-r}
		{\bf{1}}_{\{r=k,\alpha=0\}}
		\,.
	\end{align*}
	We can rewrite this, using the notation~\eqref{e.mathring} and substituting~$q=p-k$,, as 
	\begin{align}
		\label{e.correqs}
		-\nabla_y \cdot \a \nabla_y \chi_{q,\alpha,k}
		&
		=
		\nabla_y \cdot \a \nabla_x \chi_{q-1,\alpha,k} 
		+
		\sum_{i,j=1}^d \partial_{x_i} \bigl( \a_{ij} \chi_{q-1,\alpha-e_j,k} \bigr)
		\notag \\ & \quad
		+ 
		\nabla_x \cdot  \mathring{\mathbf{f}}_{q-1,\alpha,k}
		+
		\sum_{i=1}^d
		\mathring{\mathbf{f}}_{q-1,\alpha-e_i,k}
		-
		W \bigl( 
		\chi_{q-2,\alpha,k} -  {\bf{1}}_{\{ q=2, |\alpha| = 0 \}} \bigr)
		\notag \\ & \quad
		+
		\sum_{r=|\alpha|}^{q+k-2}
		\mu_{q+k-2-r}
		\bigl( \chi_{r-k,\alpha,k} - {\bf{1}}_{\{ r=k, |\alpha| = 0 \}} \bigr)
		\,.
	\end{align}
	This is the sequence of corrector equations we have been seeking. Observe that this equation involves the constants $\{ \mu_k \,: \, k\in \{ 0,\ldots,q+k-m-2\}$, which are a priori unknown. This is because these corrector equations have to be understood as \emph{coupled to the macroscopic equations}, which we introduce next. 
	Define the homogenized coefficients by
	\begin{equation*}
		\overline{\a}_{q,\alpha,k}
		:=
		\langle
		\mathbf{f}_{q,\alpha,k}
		\rangle\,.
	\end{equation*}
	We note that~$\overline{\a}_{q,\alpha,k}$ is~$\R^d$--valued and depends on the macroscopic variable~$x$. 
	Assuming for the moment that the corrector equation~\eqref{e.correqs} is satisfied, we insert it back into the expression~\eqref{e.expansa} to obtain
	\begin{align}
		\label{e.macroscopic.derivation}
		\lefteqn{
			\bigl( -\nabla \cdot \a_\ep \nabla + W - \lambda_\ep \bigr) u_\ep 
		} \quad & 
		\notag \\ & 
		=
		\sum_{p=0}^\infty
		\sum_{k=0}^p
		\sum_{m=0}^{p-k}
		\sum_{|\alpha|=m} \!\!
		\ep^{p-2} \partial_x^\alpha U_k 
		\biggl( 
		- \nabla_x \cdot  \overline{\a}_{p-1-k,\alpha,k}
		-
		\sum_{i=1}^d
		\overline{\a}_{p-1-k,\alpha-e_i,k}
		\notag \\ & 
		\qquad \qquad \qquad \qquad \qquad 
		\qquad \qquad
		+
		W
		{\bf{1}}_{p= k+2,|\alpha|=0} 
		-
		\sum_{r=m}^{p-2}
		\mu_{p-2-r}
		{\bf{1}}_{r = k,|\alpha|=0} 
		\biggr) 
		\notag \\ & 
		=
		-\nabla_x \cdot 
		\sum_{p=0}^\infty
		\sum_{k=0}^p
		\sum_{m=0}^{p-k}
		\sum_{|\alpha|=m} \!\!
		\ep^{p}  
		\overline{\a}_{p+1-k,\alpha,k} \,
		\partial_x^\alpha U_k 
		+
		\sum_{p=0}^\infty \ep^p W U_p
		-
		\sum_{p=0}^\infty\sum_{k=0}^p
		\ep^p \mu_{p-k} U_k
		\notag \\ & 
		= 
		\sum_{p=0}^\infty 
		\ep^p
		\biggl( 
		-\nabla_x \cdot 
		\Bigl( \sum_{k=0}^p
		\sum_{m=0}^{p-k}
		\sum_{|\alpha|=m} \!\!
		\overline{\a}_{p+1-k,\alpha,k} \,
		\partial_x^\alpha U_k 
		\Bigr)
		+
		W U_p
		-
		\sum_{k=0}^p
		\mu_{p-k} U_k
		\biggr)
		\,.
	\end{align}
	This is the macroscopic equation, or to be more precise, it encodes a sequence of macroscopic equations---one for every $p\in\N_0$:
	\begin{equation}
		\label{e.macroscopic.p}
		-\nabla_x \cdot 
		\biggl( \sum_{k=0}^p
		\sum_{m=0}^{p-k}
		\sum_{|\alpha|=m} \!\!
		\overline{\a}_{p+1-k,\alpha,k} \,
		\partial_x^\alpha U_k 
		\biggr)
		+
		W U_p
		=
		\sum_{k=0}^p
		\mu_{p-k} U_k
		\,.
	\end{equation}
	If we can find~$\{ \mu_k\}$ and~$\{ U_k\}$ solving the system~\eqref{e.macroscopic.p}, with correctors~$\chi_{q,\alpha,k}$ solving~\eqref{e.correqs}, then we will be able to show that the function~$u_\ep$ is close to a true eigenfunction of the operator~$\mathcal{L}_\ep$, with eigenvalue close to~$\lambda_\ep$.
	
	\smallskip
	
	Before we consider the hierarchy of equations in~\eqref{e.macroscopic.p} in more detail, we make some remarks about the first few homogenized coefficients. First, for every~$k\in\N$ and $j\in\{1,\ldots,d\}$, the corrector~$\chi_{1,e_j,k}$ is the usual first-order corrector in homogenization theory. In particular, it is independent of~$x$ and solves the equation
	\begin{equation*}
		-\nabla \cdot \a (e_j +\nabla \chi_{1,e_j,k})  = 0\,.
	\end{equation*}
	We deduce that, for each~$k\in\N$, the~$(i,j)$th entry of the usual homogenized matrix~$\ahom$ in elliptic homogenization theory is equal to the~$i$th component of the coefficient~$\ahom_{1,e_j,k}$ defined above (which in particular does not depend on~$x$):
	\begin{equation*}
		\ahom_{ij} = \bigl( \ahom_{1,e_j,k} \bigr)_i\,.
	\end{equation*}
	Recalling also that~$\mu_0 = \lambda_0$, we may therefore write~\eqref{e.macroscopic.p} as
	\begin{equation}
		\label{e.macroscopic.p.aspert}
		-\nabla \cdot \ahom \nabla U_p
		+ (W- \lambda_0) U_p 
		=
		\nabla_x \cdot 
		\biggl( \sum_{k=0}^{p-1}
		\sum_{m=0}^{p-k}
		\sum_{|\alpha|=m} \!\!
		\overline{\a}_{p+1-k,\alpha,k} \,
		\partial_x^\alpha U_k 
		\biggr)
		+
		\sum_{k=0}^{p-1}
		\mu_{p-k} U_k
		\,.
	\end{equation}
	This equation gives us hope that we can  solve for~$U_p$, provided that we have already determined~$U_0,\ldots,U_{p-1}$ and~$\mu_1,\ldots,\mu_{p}$ as well as~$\ahom_{q,\alpha,k}$ for every~$(q,\alpha,k)$ with~$2 \leq q+k \leq p+1$ and~$0\leq |\alpha| \leq q-1$. 
	We will require that~$U_k$ be orthogonal to~$U_0$ in~$L^2(\R^d)$ for every~$k\geq 1$:
	\begin{equation}
		\label{e.Uk.U0.perp}
		\int_{\R^d} U_k(x) U_0(x)\,dx = 0\,, \quad \forall k\geq 1\,.
	\end{equation}

	\smallskip
	
	We can determine the values of~$\{ \mu_k \}$ be requiring that~\eqref{e.macroscopic.p.aspert} be solvable; that is, the right side of~\eqref{e.macroscopic.p.aspert} must be orthogonal to~$U_0$. This yields a formula for~$\mu_p$:
	\begin{equation*}
		\mu_p = 
		\int_{\R^d} 
		\mu_p U_0^2
		=
		\sum_{k=0}^{p-1} 
		\mu_{p-k} 
		\int_{\R^d} 
		U_k U_0
		=
		\sum_{k=0}^{p-1}
		\sum_{m=0}^{p-k}
		\sum_{|\alpha|=m} \!\!
		\int_{\R^d}
		\overline{\a}_{p+1-k,\alpha,k} \,
		\partial_x^\alpha U_k 
		\cdot \nabla_x U_0
		\,.
	\end{equation*}
	This gives a formula for~$\mu_p$ in terms of~$U_0,\ldots,U_{p-1}$ and the tensors~$\ahom_{q,\alpha,k}$ with indices~$(q,\alpha,k)$ satisfying~$2 \leq q+k \leq p+1$ and~$0\leq |\alpha| \leq q-1$.
	
	\subsection{Rigorous construction with estimates}
	\label{ss.rigorous.simple}

	In this subsection we use the foregoing formal argument to provide a rigorous inductive proof to show that the higher order correctors $\{\bchi_{q,m,k}\}$ and effective tensors~$\{ \bba_{q,m,k}\}$ are well-defined for~$q,m,k \in \mathbb{N}_0.$ This amounts to showing that we can solve the corrector equations in some appropriate order so that all the terms on the right side have been already previously defined. 
	
	\smallskip
	
	In this section we employ tensor notation instead of the multi-index notation from the previous section. All tensors are actually indexed by these multi-indices, and the precise meaning of the (implicit) tensor contractions (sometimes denoted by~``$:$'') can be inferred from the computations in the previous section. However, these turn out to be not very important for the computations that follows, and so for convenience we use the more compact tensor notation. 
	
	\smallskip
	
	As in the rest of this section, we assume that~$\lambda_0$ is a simple eigenvalue of the homogenized operator $\mathcal{L}_0 := -\nabla \cdot \bba \nabla  + W$ and~$U_0$ is a corresponding eigenfunction such  that $\| U_0\|_{L^2(\RR^d)} = 1.$ 
	
	\smallskip
	
	\textit{Base case.} 
	We initialize the construction by making the following definitions for the first few correctors and effective parameters:
	\begin{itemize}
		\item We set $\mu_0 :=\lambda_0$.
		\item We take $\bchi_{0,0,k} := 1$, as well as $\bbf_{0,0,k}:=0$ and $\bba_{0,0,k}:=0$, for every $k\in\mathbb{N}_0$.
		
		\item $\bchi_{q,0,k} := 0$, as well as $\bbf_{q,0,k}:=0$ and $\bba_{q,0,k}:=0$, for every $q,k\in\mathbb{N}$, $k\in\mathbb{N}_0$.
		
		\item $\bchi_{q,m,k} := 0$, as well as $\bbf_{q,m,k}:=0$ and $\bba_{q,m,k}:=0$, for every $q,m,k\in\mathbb{N}_0$ with $m>q$.
		
		\item We define~$\bchi_{1,1,k}$, for each $k\in \mathbb{N}_0$, to be the usual first-order corrector in classical periodic homogenization, that is, the solution of
		\begin{equation} \label{e.hocorr11} 
			\left\{
			\begin{aligned}
				&-\nabla_y \cdot \bha \nabla_y \bchi_{1,1,k} = \nabla_y \cdot (\bha \otimes 1)  \quad \mbox{ in } \mathbb{T}^d,\\
				& \langle \bchi_{1,1,k} \rangle = 0.
			\end{aligned} \right.
		\end{equation}
		We also define 
		\begin{equation*}
			\bbf_{1,1,k}:= \bha (I+\nabla_y\bchi_{1,1,k})
			\quad \mbox{and} \quad
			\bba_{1,1,k}:=\langle \bbf_{1,1,k} \rangle.
		\end{equation*}
		Observe that $\bba_{1,1,k}$ corresponds to the usual homogenized matrix in classical homogenization~$\ahom$
		Note that~$\bchi_{1,1,k}$,~$\bbf_{1,1,k}$ and~$\bba_{1,1,k}$ depend
		neither on the index~$k$ nor the slow variable~$x$.

		\item We define~$\bchi_{2,1,k} := 0$, as well as $\bbf_{2,1,k}:=0$ and $\bba_{2,1,k}:=0$, for every $k\in\mathbb{N}_0$.
		
		\item We define $\bchi_{2,2,k}$, for each $k\in \mathbb{N}_0$, to be the solution of
		\begin{equation} \label{e.hocorr1} \left\{
			\begin{aligned}
				&-\nabla_y \cdot \bha \nabla_y \bchi_{2,2,k} = \nabla_y \cdot (\bha \otimes \bchi_{1,1,k}) +  \bha \otimes 1 + \bha \nabla_y \bchi_{1,1,k} -\bba_{1,1,k} \quad \mbox{ in } \mathbb{T}^d,\\
				& \langle \bchi_{2,2,k} \rangle = 0.
			\end{aligned} \right.
		\end{equation}
		We also define 
		\begin{equation*}
			\bbf_{2,2,k}:= \bha \nabla_y \bchi_{2,2,k} + \bha\otimes \bchi_{1,1,k}  
			\quad \mbox{and} \quad
			\bba_{2,2,k}:=\langle \bbf_{2,2,k}  \rangle.
		\end{equation*}
		Note that~$\bchi_{2,2,k}$ is the second-order corrector, and $\bba_{2,2,k}$ is the usual third-order homogenized matrix, in classical homogenization. 
		In particular,~$\bchi_{2,2,k}$,~$\bbf_{2,2,k}$ and~$\bba_{2,2,k}$ depend neither on the index~$k$ nor the slow variable~$x$, and the symmetric part of~$\bba_{2,2,k}$ vanishes. 
		
	\end{itemize}

	\textit{Induction step.} 
	Let us suppose that, for some integer~$K\in\mathbb{N}$, $K\geq 2$, we have defined~$\bchi_{q,m,k},$ the associated fluxes~$\bbf_{q,m,k}$ and homogenized coefficients~$\bba_{q,m,k}$ for indices
	\begin{equation} \label{e.qmk.induction}
		(q,m,k) \in 
		J(K)
		:=
		\bigl\{ (q,m,k) \,:\, m \in \mathbb{N}_0, \,  0 \leq k \leq K, \,  0 \leq q \leq K+2-k \bigr\}
		\,,
	\end{equation}
	as well as~$\mu_k$ and $U_k$ for every $k\in\{0,\ldots,K-2\}$.
	Note that for $K=2$ we defined these objects in the base case above. 
	
	\smallskip
	
	We then make the following definitions.
	
	\begin{itemize}
		\item We define $\mu_{K-1}$ by
		\begin{equation} \label{e.muk-1def}
			\mu_{K-1} := 
			\sum_{k=0}^{K-2}
			\sum_{m={1}}^{K-k} 
			\int_{\RR^d} \bba_{K-k,m,k} : \nabla^m U_k : \nabla U_0 
			\,.
		\end{equation}
		
		\item We define the macroscopic function $U_{K-1}$ to be the unique solution of 
		\begin{equation}
			\label{e.Uk1.def}
			(\mathcal{L}_0 - \mu_0)U_{K-1} = \sum_{k=0}^{K-2}
			\mu_{K-1-k}U_k
			+
			\nabla \cdot 
			\sum_{k=0}^{K-2}
			\sum_{m={1}}^{K-k} \bba_{K-k,m,k}:\nabla^m U_k
			\quad \mbox{in} \ \RR^d \,.
		\end{equation}
		which is orthogonal to $U_0$ in $L^2(\RR^d)$. Note that this indeed uniquely determines $U_{K-1}$ since~$\mu_{K-1}$ was chosen above so that the right side of~\eqref{e.Uk1.def} is orthogonal to the eigenspace of $\mathcal{L}_0$ corresponding to $\mu_0$. 
		
		\item The functions $\bchi_{2,m,k}$, as well as $\bbf_{2,m,k}$ and $\bba_{2,m,k}$, have already been defined in the base case above, for every $k$ and in particular for $k=K+1$. The only nonzero function among these is $\bchi_{2,2,k}$ which was defined in~\eqref{e.hocorr1}. 
		
		\item We define $\bchi_{q,m,k}$ 
		for each~$(q,m,k) \in J(K+1)\setminus J(K),$  $q\neq 2$, to be the solution of  
		\begin{equation} \label{e.hocorrq} \left\{
			\begin{aligned}
				&    -\nabla_y \cdot \bha \nabla_y \bchi_{q,m,k}&&  \!\!  \!\!  = \nabla_y \cdot (\bha \otimes \bchi_{q-1,m-1,k})  + \nabla_y \cdot \bha\nabla_x \bchi_{q-1,m,k}  
				\\ &&& \qquad      
				+ \mathring{\bbf}_{q-1,m-1,k}  + \nabla_x \cdot \mathring{\bbf}_{q-1,m,k} 
				\\ &&& \qquad  
				- W(x) \mathring{\bchi}_{q-2,m,k} 
				+ \sum_{r=m+k}^{q-2+k} \mu_{q-2+k-r} \mathring{\bchi}_{r-k,m,k} \quad \quad \mbox{ in } \mathbb{T}^d,\\
				& \ \    \langle \bchi_{q,m,k} \rangle = 0.
			\end{aligned} \right.
		\end{equation}
		Note that~$q\neq 2$ implies~$q\geq 3$ and thus~$k\leq K$. Therefore all the terms on the right side have been defined already, by the induction hypothesis, because all the index triples belong to~$J(K)$
		and the highest index~$i$ of~$\mu_i$ that appears in \eqref{e.hocorrq} is~$i=K-1$, which we have above defined in~\eqref{e.muk-1def}. Moreover, the right side of the equation has zero mean and so the equation is uniquely solvable. 
		We then define 
		\begin{equation} \label{e.bbf.rig}
			\bbf_{q,m,k} := \bha \nabla_x \bchi_{q-1,m,k} + \bha \otimes \bchi_{q-1,m-1,k} + \bha \nabla_y \bchi_{q,m,k},
		\end{equation}
		and then
		\begin{equation} \label{e.aqmk.rig}
			\bba_{q,m,k}(x) := \langle \bbf_{q,m,k}(x,\cdot) \rangle.
		\end{equation}
		
	\end{itemize}
	We have therefore defined $\bchi_{q,m,k},$ the associated fluxes $\bbf_{q,m,k}$ and homogenized coefficients $\bba_{q,m,k}$ for every $(q,m,k)\in J(K+1)$. By induction, this concludes the construction of the correctors, homogenized tensors.

	\subsection{Regularity estimates}
	
	In this section we study the regularity, with respect to the slow variable~$x$, of the objects defined above in Section~\ref{ss.rigorous.simple}. 
	This amounts to going over the entire recursive construction, step-by-step, and estimating all~$x$ derivatives of each newly defined object. This is a rather laborious and tedious but straightforward process. 
	
	\smallskip

	The regularity of the effective homogenized tensors~$\bba_{q,m,k}$ will follow from the regularity of the fluxes~$\bbf_{q,m,k}$, and the latter will be obtained rather easily from the product rule and the regularity in~$x$ of the correctors~$\bchi_{q,m,k}$. The latter will be obtained by repeatedly differentiating the equation for the correctors and using the regularity of all objects previously defined. Fortunately, every term in~\eqref{e.hocorrq}, with one exception, has only one factor with~$x$ dependence. The exception is the term $W(x) \bchi_{q-2,m,k}$, which is simple to differentiation and estimate. Therefore the computation is not overly involved. 
	
	We begin with a preliminary lemma that we repeatedly invoke in our bounds for $\bchi_{q,m,k},$ and hence the corrections $\bba_{q,m,k}$, $\mu_k$ and $U_k.$ 
	
	\begin{lemma}
		\label{l.cacciop.per}
		Let $\Phi:\RR^d \times \mathbb{T}^d \to \RR$ be the unique periodic (in $y$) solution to 
		\begin{align}
			-\nabla_y \cdot \bha \nabla_y \Phi = \nabla_y \cdot F + G, \quad \quad \langle \Phi(x,0) \rangle = 0,
		\end{align}
		where $F, G \in H_{per}^1(\mathbb{T}^d, C^\infty(\RR^d))$ with $\langle G(x,\cdot) \rangle = 0$ for every $x.$ Then for every $x \in \RR^d$ 
		\begin{align}
			\|\nabla_y \partial_x^\alpha \Phi(x,\cdot)\|_{L^2(\mathbb{T}^d)} \leqslant  C\bigl( \|\partial_x^\alpha F(x,\cdot)\|_{L^2(\mathbb{T}^d)} + \|\partial_x^\alpha G(x,\cdot)\|_{L^2(\mathbb{T}^d)}\bigr),
		\end{align}
		for a universal constant $C (\theta,d) > 0$ , and for any multiindex $\alpha \in \NN_0^d.$
	\end{lemma}
	\begin{proof}
		We set $v:= \partial_x^\alpha \Phi,$ for any $\alpha \in \mathbb{N}_0^d.$ Then, it is clear that $\langle v(x,\cdot) \rangle = 0$ by the choice of normalization in $\Phi.$ Moreover, $v$ satisfies 
		\begin{equation}
			\begin{aligned}
				-\nabla_y \cdot \bha \nabla_y v = \nabla_y\cdot \partial_x^\alpha F + \partial_x^\alpha G.
			\end{aligned}
		\end{equation}
		Multiplying the equation by $v$, integrating by parts on $\mathbb{T}^d$, using the ellipticity of $\bba$ and Cauchy-Schwarz and Poincar\'{e} inequalities, we obtain the desired estimate. 
	\end{proof}

	We are now ready to prove regularity estimates for each object constructed  in Section~\ref{ss.rigorous.simple}.
	
	\begin{proposition} \label{p.exp.bounds}
		There exists $C(d,\theta)<\infty$
		such that, for every~$q,m,k\in\NN$ with~$m\leq q$,
		\begin{equation}
			\label{e.reg.indyhyp}
			\left\{
			\begin{aligned}
				& |\mu_k| 
				\leq 
				\frac{\lambda^{\frac{3k}2}}{{\gamma(\lambda)^{k-1}}}
				\exp \bigl( C^{k+1} \bigr) \,,\\
				& \nnn U_k \nnn_{\lambda,C^k}
				\leq 
				\frac{\lambda^{\frac{3k}2}}{{\gamma(\lambda)^{k}}}
				\exp \bigl( C^{k+1} \bigr) \,, \\
				& \frac1{(q+l)!}\sup_{x \in \RR^d} 
				\bigl( \lambda +|x|^2\bigr)^{-\frac 12(q-l)}
				\| \nabla_x^l \bchi_{q,m,k}(x,\cdot) \|_{{H^1}(\mathbb{T}^d)} 
				\leq 
				{\frac{\exp ( C^{q-m+k} )}{\gamma(\lambda)^{(q-2-m)_+}}}
				\,,
			\end{aligned}
			\right.
		\end{equation}
		where, we recall that $\gamma(\lambda)$ is the spectral gap of the simple eigenvalue $\lambda. $ 
	\end{proposition}
	\begin{proof}
		We argue by induction, following the same procedure as in the construction of the objects.
		
		\smallskip
		
		\emph{Step 1.} 
		We begin by noticing that each of the objects introduced in the ``base case'' of the construction; in particular, for every $(q,m,k) \in J(2)$ and $(q,m,k) \in \{ (2,2,k) \,:\, k\in\NN\}$, we have that estimates~\eqref{e.reg.indyhyp} are satisfied. 
		
		\smallskip
		
		Turning to the the induction step, we suppose that~$A,B\in[1,\infty)$ and~$K\in\NN$ are such that,
		for every~$(q,m,k) \in J(K)$,
		\begin{equation}
			\label{e.reg.indyhyp.hyp}
			\left\{
			\begin{aligned}
				& |\mu_k| 
				\leq 
				\frac{\lambda^{\frac{3k}2}}{{\gamma(\lambda)^{k-1}}}
				\exp \bigl( A^{k+1} \bigr) \,,\\
				& \nnn U_k \nnn_{\lambda,B^k}
				\leq 
				\frac{\lambda^{\frac{3k}2}}{{\gamma(\lambda)^{k}}}
				\exp \bigl( A^{k+1} \bigr) \,, \\
				& \frac1{(q+l)!}\sup_{x \in \RR^d} 
				\bigl( \lambda +|x|^2\bigr)^{-\frac 12(q-l)}
				\| \nabla_x^l \bchi_{q,m,k}(x,\cdot) \|_{{H^1}(\mathbb{T}^d)} 
				\leq 
				{\frac{\exp ( A^{q-m+k} )}{\gamma(\lambda)^{(q-2-m)_+}}}
				\,,
			\end{aligned}
			\right.
		\end{equation}
		{with the last estimate  holding for every $l \in \NN_0$.}
		Recall~$J(K)$ is defined in~\eqref{e.qmk.induction}.
		We will show that, if~$A$ and~$B$ are chosen sufficiently large, depending only on~$(d,\theta)$, then the same estimates are valid for~$(q,m,k) \in J(K+1)$.
		
		\smallskip
		
		\emph{Step 2.}
		We record estimates for $\bbf_{q,m,k}$. The claim is that, for every $(q,m,k) \in J(K)$, and for every $l \in \NN_0,$
		\begin{equation}
			\label{e.fbb.ind.bound}
			\frac{1}{(q+l)!}
			\sup_{x\in\RR^d} (\lambda+|x|^2)^{-\frac12(q-l)} \| \nabla_x^l \bbf_{q,m,k}(x,\cdot)\|_{L^2(\TT^d)}
			\leq
			\frac{1}{\gamma(\lambda)^{(q-2-m)_+}}\exp \bigl( A^{q-m+k} \bigr)
			\,.
		\end{equation}
		Compute, using the induction hypothesis, for every $(q,m,k)\in J(K)$, 
		\begin{align*}
			\lefteqn{
				\| \nabla_x^l \bbf_{q,m,k}(x) \|_{L^2(\TT^d)}
			} \qquad & 
			\\ &
			\leq
			\| 
			\bha \nabla_x^{l+1} \bchi_{q-1,m,k}
			\|_{L^2(\TT^d)}
			+
			\| 
			\bha \otimes 
			\nabla_x^l\bchi_{q-1,m-1,k}
			\|_{L^2(\TT^d)}
			+
			\| 
			\bha \nabla_y \nabla_x^l \bchi_{q,m,k}
			\|_{L^2(\TT^d)}
			\\ & 
			\leq
			(q+l)! \frac{1}{\gamma(\lambda)^{(q-3-m)_+}} (\lambda+|x|^2)^{\frac12(q-l-2)} 
			\exp \bigl( A^{q-m+k-1} \bigr)
			\\ & \qquad 
			+ (q+l-1)!\frac{1}{\gamma(\lambda)^{(q-2-m)_+}} (\lambda+|x|^2)^{\frac12(q-l-1)} 
			\exp \bigl( A^{q-m+k} \bigr)
			\\ & \qquad 
			+
			(q+l)!\frac{1}{\gamma(\lambda)^{(q-2-m)+}}(\lambda+|x|^2)^{\frac12(q-l)} 
			\exp \bigl( A^{q-m+k} \bigr)
			\\ & 
			\leq 
			C (q+l)! \frac{1}{\gamma(\lambda)^{(q-2-m)_+}}(\lambda+|x|^2)^{\frac12(q-l)}  
			\exp \bigl( A^{q-m+k} \bigr)
			\,.
		\end{align*}
		This yields~\eqref{e.fbb.ind.bound}; moreover, by the definition in \eqref{e.aqmk.rig}, the homogenized tensors then satisfy 
		\begin{equation}
			\label{e.aqmk.rig.bound}
			\frac{1}{(q+l)!} \sup_{x \in \RR^d}(\lambda + |x|^2)^{-\frac12 (q-l)} |\nabla_x^l \bba_{q,m,k}(x)| \leqslant C\frac{1}{\gamma(\lambda)^{(q-2-m)_+}} \exp(A^{q-m+k})\,.
		\end{equation}

		\smallskip
		
		\emph{Step 3.} The estimate for $\mu_{K-1}$. By our induction hypothesis, we have 
		\begin{align*}
			\biggl| 
			\sum_{k=1}^{K-2} \mu_{K-1-k}\int_{\RR^d} U_k U_0
			\biggr|
			&
			\leq
			\sum_{k=1}^{K-2}
			|\mu_{K-1-k}| 
			\| U_k \|_{L^2(\RR^d)}
			\| U_0 \|_{L^2(\RR^d)}
			\\ & 
			\leq 
			\sum_{k=1}^{K-2} \frac{\lambda^{\frac32 (K-1-k) + \frac{3k}{2}}}{\gamma(\lambda)^{K-1-k-1+k}}
			\exp\bigl( A^{K-1-k} + A^k) 
			\\ & 
			\leq
			C\frac{\lambda^{\frac32 (K-1)}}{\gamma(\lambda)^{K-2}} (K-2) \exp \bigl( A^{K-2} \bigr)
			\,.
		\end{align*}
		We recall by the energy bound that $\|\nabla U_0\|_{L^2(\RR^d)} \leq \sqrt{\lambda}$ and next estimate
		\begin{align*}
			\lefteqn{
				\biggl|
				\sum_{k=0}^{K-2}
				\sum_{m={1}}^{K-k} 
				\int_{\RR^d} \bba_{K-k,m,k} : \nabla^m U_k : \nabla U_0
				\biggr|
			} \quad & \notag 
			\\ & 
			\leq
			\|\nabla U_0\|_{L^2(\RR^2)}
			\sum_{k=0}^{K-2}
			\sum_{m={1}}^{K-k}
			\biggl( \int_{\RR^d} 
			|\bba_{K-k,m,k}|^2 |\nabla^m U_k|^2 \biggr)^{\sfrac12}  
			\notag \\ & 
			\leq 
			C\sqrt{\lambda}
			\sum_{k=0}^{K-2}
			\sum_{m={1}}^{K-k}
			(K{-}k)! \frac{\exp(A^{K-m})}{\gamma(\lambda)^{(K-k-m-2)_+}}
			\biggl( \int_{\RR^d} 
			(\lambda{+}|x|^2)^{K-k} |\nabla^m U_k(x)|^2\,dx \biggr)^{\sfrac12} 
			\notag \\ & 
			\leq 
			C\sqrt{\lambda}
			\sum_{k=0}^{K-2}
			\sum_{m={1}}^{K-k}  
			({B^k})^{K-k+m} \Lambda_2^{m}
			(K{-}k)!^2 m! \frac{\exp(A^{K-m})}{\gamma(\lambda)^{(K{-}k{-}m{-}2)_+}}
			\nnn U_k \nnn_{\lambda,{B^k}}
			\notag \\ & 
			\leq
			C{\lambda}^{1 + \frac32(K-2)}
			\sum_{k=0}^{K-2}
			\sum_{m={1}}^{K-k}
			(K{-}k)!^3 \Lambda_2^m \frac{1}{\gamma(\lambda)^{(K-k-m-2)_+ {+} k}}
			\\ & \qquad \qquad \qquad \qquad \qquad  
			\times \exp\bigl(A^{K-m}{+}A^k {+} {k(K{-}k{+}m)\log B {+} m \log \Lambda_2 }\bigr)
			\notag \\ & 
			\leq
			\frac{C{\lambda}^{\frac32(K-1)}}{\gamma(\lambda)^{K-2}}
			K^2 (K!)^3
			\exp\bigl(A^{K-1}{+}A^{K-2} {+} {K^2 \log B {+} K\log \Lambda_2}\bigr)
			\,.
		\end{align*}
		Using the triangle inequality and~\eqref{e.muk-1def} and  choosing $B$ sufficiently large depending on $A,$ we have that
		\begin{align*}
			|\mu_{K-1}| \leq \frac{C{\lambda}^{\frac32(K-1)}}{\gamma(\lambda)^{K-2}}\exp(A^K)\,.
		\end{align*}

		\smallskip
		
		\emph{Step 4.} 
		We estimate $\nnn \nabla \cdot( \bba_{K-k,m,k}:\nabla^m U_k ) \nnn_{\lambda,\Theta}$ for each $\Theta >0$.
		We will show that, for each $m\geq 2$, 
		\begin{align}
			\label{e.nnn.bbaUk}
			&
			\nnn \nabla \cdot( \bba_{K-k,m,k}:\nabla^m U_k ) \nnn_{\lambda,C\Theta}
			\notag \\ & \qquad 
			\leq
			\frac{1}{\gamma(\lambda)^{(K-k-2-m)_+}}\exp( A^{K-m} )
			\nnn U_k \nnn_{\lambda,\Theta}
			(C\Theta )^{K-k} \Lambda_2^m
			(K-k+m)!
		\end{align}
		We have that
		\begin{align*}
			\nabla^{l+1} ( \bba_{K-k,m,k}:\nabla^m U_k )
			&
			=
			\sum_{j=0}^{l+1}
			\binom{l+1}{j} \nabla^j \bba_{K-k,m,k}:\nabla^{l+m+1-j} U_k
			.
		\end{align*}
		We estimate the $\nnn\cdot\nnn_{\lambda,\Theta}$ norm of each term on the right side: by \eqref{e.aqmk.rig.bound} in Step~2  we have
		\begin{align*}
			\lefteqn{
				\biggl( 
				\int_{\RR^d}
				(\lambda+|x|^2)^{n}
				\bigl| 
				\nabla^j \bba_{K-k,m,k} \bigr|^2
				\bigl|\nabla^{l+m+1-j} U_k
				\bigr|^2
				\exp\bigl(  (\alpha|x|^2-c_2\lambda)_+\bigr)\,dx
				\biggr)^{\sfrac12}
			} \quad & 
			\\ & 
			\leq
			\frac{1}{\gamma(\lambda)^{(K-k-m-2)_+}}\exp( A^{K-m} )
			\\ & \qquad \qquad \qquad 
			\times
			\biggl( 
			\int_{\RR^d}
			(\lambda+|x|^2)^{n+K-k-j}
			\bigl|\nabla^{l+m+1-j} U_k \bigr|^2
			\exp\bigl( (\alpha|x|^2-c_2\lambda)_+\bigr)\,dx
			\biggr)^{\sfrac12}
			\\ & 
			\leq
			\frac{1}{\gamma(\lambda)^{(K-k-m-2)_+}}\exp( A^{K-m} )
			\\ & \qquad \qquad \qquad 
			\times
			\nnn U_k \nnn_{\lambda,\Theta}
			\Theta^{n+K+l+m+1-k-2j}
			\Lambda_2^{l+m+1-j}
			(n{+}K{-}k{-}j)!(l{+}m{+}1{-}j)!
			\,.
		\end{align*}
		Substituting this estimate into the identity above and using the triangle inequality,
		we obtain
		\begin{align*}
			\lefteqn{
				\biggl( 
				\int_{\RR^d}
				(\lambda+|x|^2)^{n}
				\bigl| \nabla^{l+1} ( \bba_{K-k,m,k}:\nabla^m U_k )
				\bigr|^2
				\exp\bigl( (\alpha|x|^2-c_2\lambda)_+\bigr)\,dx
				\biggr)^{\sfrac12}
			} \quad & 
			\\ & 
			\leq  
			\sum_{j=0}^{l+1}
			\binom{l+1}{j}\biggl( 
			\int_{\RR^d}
			(\lambda+|x|^2)^{n}
			\bigl| 
			\nabla^j \bba_{K-k,m,k} \bigr|^2
			\bigl|\nabla^{l+m+1-j} U_k
			\bigr|^2
			\exp\bigl( (\alpha|x|^2-c_2\lambda)_+\bigr)\,dx
			\biggr)^{\sfrac12}
			\\ &  
			\leq
			\frac{1}{\gamma(\lambda)^{(K-k-m-2)_+}}\exp( A^{K-m} )
			\nnn U_k \nnn_{\lambda,\Theta}
			(C\Theta)^{n+K+l+m+1-k}
			\Lambda_2^{l+m+1}
			(K-k+m)! \, n! \, l!
			\,.
		\end{align*}
		Taking the supremum over $n$ and $l$ yields~\eqref{e.nnn.bbaUk}.

		\smallskip
		
		\emph{Step 5.} The estimate for $U_{K-1}$.
		We apply Lemma~\ref{l.reg} to the equation~\eqref{e.Uk1.def} for~$U_{K-1}$ and use the triangle inequality to get
		\begin{align} 
			\label{e.UK-1Step5}
			\lefteqn{
				\nnn U_{K-1} \nnn_{\lambda,C\Theta}
			} \qquad & 
			\notag \\ & 
			\leq 
			\gamma(\lambda)^{-1} \biggl(
			\sum_{k=0}^{K-2}
			|\mu_{K-1-k}|
			\cdot
			\nnn 
			U_{k}
			\nnn_{\lambda,\Theta}
			+
			\sum_{k=0}^{K-2}
			\sum_{m=1}^{K-k}
			\nnn
			\nabla \cdot ( \bba_{K-k,m,k}:\nabla^m U_k ) 
			\nnn_{\lambda,C\Theta} \biggr)
			\,.
		\end{align}
		{Recall in the induction step that we must estimate $U_{K-1}$ in the $\nnn \cdot \nnn_{\lambda,C\Theta}$ norm with the choice $C = B, \Theta = B^{K-2}$ so that $C \Theta = B^{K-1}.$}
		We estimate the second term on the right side by 
		\begin{align*}
			\lefteqn{
				\sum_{k=0}^{K-2}
				\sum_{m=1}^{K-k}
				\nnn
				\nabla \cdot ( \bba_{K-k,m,k}:\nabla^m U_k ) 
				\nnn_{\lambda,{CB^k}}
			} \qquad & 
			\\ & 
			\leq
			\sum_{k=0}^{K-2}
			\sum_{m=1}^{K-k}\frac{1}{\gamma(\lambda)^{(K-k-m-2)_+}}
			\exp( A^{K-m} )
			\nnn U_k \nnn_{\lambda,{B^k}}
			(C{B^k} )^{K-k}
			\Lambda_2^{m}
			(K-k+m)!
			\\ & 
			\leq
			\sum_{k=0}^{K-2} \frac{1}{\gamma(\lambda)^{(K-k-3)_+ + k}}
			(C{B^k} )^{K-k} {\lambda^{\frac{3k}{2}}}
			\exp( A^{K-1} +A^{k+1})
			(K-k+1)!
			\\ & 
			\leq
			K (K+1)! \frac{{\lambda}^{\frac32(K-2)}}{\gamma(\lambda)^{K-2}}
			(C{B^{\frac{K^2}{4}}} )
			\exp( A^{K-1} +A^{K-1})
			\\ & 
			\leq 
			\frac{{\lambda}^{\frac32(K-2)}}{\gamma(\lambda)^{K-2}} \exp\biggl(2 A^{K-1} + (K+2)\log (K+2) + \tfrac{K^2}{4} \log B \biggr)\,.
		\end{align*}
		In the above, we tacitly used the inequality that $\nnn \cdot \nnn_{\lambda,B^{K-1}} \leq \nnn \cdot \nnn_{\lambda, B^k}$ for any $k \leq K-1.$ Once more, choosing $B$ suitably large in terms of $A,$ this completes the induction step for estimating $\nnn U_{K-1}\nnn_{\lambda, B^{K-1}},$ since from \eqref{e.UK-1Step5} using the triangle inequality we get 
		\begin{equation*}
			\nnn U_{K-1}\nnn_{\lambda, B^{K-1}} \leqslant  \frac{\lambda^{\frac32(K-2)}}{\gamma(\lambda)^{K-1}} \exp (A^K)\,. 
		\end{equation*}

		\emph{Step 6.} The estimates for $\bchi_{q,m,k}$ for $(q,m,k)\in J(K+1)$. By Lemma~\ref{l.cacciop.per}, we have
		\begin{align*}
			\lefteqn{
				\| \nabla_y \nabla_x^l \bchi_{q,m,k}(x,\cdot) \|_{L^2(\TT^d)}
			} \qquad & 
			\\ & 
			\leq
			\|  
			\bha \otimes \nabla_x^l \bchi_{q-1,m-1,k}(x,\cdot)
			\|_{L^2(\TT^d)}
			+
			\| \bha \nabla_x^{l+1} \bchi_{q-1,m,k}(x,\cdot)
			\|_{L^2(\TT^d)}
			\\ & \qquad
			+
			\| \nabla_x^l 
			\mathring{\bbf}_{q-1,m-1,k}(x,\cdot)
			\|_{L^2(\TT^d)}
			+
			\| \nabla_x^{l+1} \mathring{\bbf}_{q-1,m,k}(x,\cdot)
			\|_{L^2(\TT^d)}
			\\ & \qquad
			+
			\| \nabla_x^l 
			( W \mathring{\bchi}_{q-2,m,k})(x,\cdot)
			\|_{L^2(\TT^d)}
			+
			\sum_{r=m+k}^{q-2+k}
			|\mu_{q-2+k-r}|
			\| \nabla_x^l 
			\mathring{\bchi}_{r-k,m,k}(x,\cdot)
			\|_{L^2(\TT^d)}
			\,.
		\end{align*}
		Using the induction hypothesis and~\eqref{e.fbb.ind.bound} we can bound the terms on the first two lines by
		\begin{equation*}
			\frac{ 
				(\lambda+|x|^2)^{-\frac12(q-l-1)} }{\gamma(\lambda)^{(q-m-2)_+}}
			\exp\bigl( A^{q-m+k} \bigr)
			\,.
		\end{equation*}
		The first term on the third line is bounded by 
		\begin{align*}
			\lefteqn{
				\| \nabla_x^l 
				( W \mathring{\bchi}_{q-2,m,k})(x,\cdot)
				\|_{L^2(\TT^d)}
			} \qquad& 
			\\ & 
			\leq
			C^l 
			\sum_{j=0}^l
			\bigl| \nabla_x^j W(x) \bigr| 
			\bigl\| \nabla_x^{l-j} \mathring{\bchi}_{q-2,m,k}(x,\cdot) \bigr\|_{L^2(\TT^d)}
			\\ & 
			\leq
			C^\ell 
			\sum_{j=0}^l
			(1+|x|^2)^{\frac12(2-j)}
			(\lambda+|x|^2)^{\frac12(q-2-l+j)}
			\frac{1}{\gamma(\lambda)^{(q-m-4)_+}}\exp\bigl( A^{q-m+k} \bigr)
			\\ & 
			\leq 
			\frac{(\lambda+|x|^2)^{\frac12(q-l)}}{\gamma(\lambda)^{(q-m-4)_+}}
			\exp\bigl( A^{q-m+k} \bigr)
			\,.
		\end{align*}
		To prepare for the estimate of the second term on the third line, we first observe that, for every $r\in \{ m+k,\ldots,q+k-2\}$,
		\begin{align*}
			\lefteqn{
				|\mu_{q-2+k-r}|
				\| \nabla_x^l 
				\mathring{\bchi}_{r-k,m,k}(x,\cdot)
				\|_{L^2(\TT^d)}
			} \qquad & 
			\\  & 
			\leq
			\frac{{\lambda}^{\frac32(q-2+k-r)}}{\gamma(\lambda)^{q-3+k-r}} \frac{1}{\gamma(\lambda)^{(r-k-m-2)_+}}\exp\bigl( A^{q+k-r-1} + A^{r-m} \bigr)
			(\lambda+|x|^2)^{\frac12(r-k-l)}
			\\ & 
			\leq
			\frac{(\lambda + |x|^2)^{\frac32 (q-2+k-r)} (\lambda+|x|^2)^{\frac32{(r-k-l)}}}{\gamma(\lambda)^{q-3-m}}
			\exp\bigl( A^{q-m-1} + A^{q-m+k-2} \bigr)\\
			& = \frac{ (\lambda + |x|^2)^{\frac32 (q-l-2)}}{\gamma(\lambda)^{q-3-m}} \exp(A^{q-m-1} + A^{q-m+k-2})
			\,,
		\end{align*}
		and then sum this over $r \in \{ m+k,\ldots,q+k-2\}$ to get 
		\begin{align*}
			\lefteqn{
				\sum_{r=m+k}^{q-2+k}
				|\mu_{q-2+k-r}|
				\| \nabla_x^l 
				\mathring{\bchi}_{r-k,m,k}(x,\cdot)
				\|_{L^2(\TT^d)}
			} \qquad & 
			\\ &
			\leq
			(q-m-2)
			\frac{(\lambda+|x|^2)^{\frac32(q-l-2)}}{\gamma(\lambda)^{q-3-m}}
			\exp\bigl( A^{q-m-1} + A^{q-m+k-2}\bigr)\\
			& \leq \frac{1}{\gamma(\lambda)^{q-3-m}} ( \lambda + |x|^2)^{\frac32(q-l-2)} \exp(A^{q-m+k})  \,.
		\end{align*}
		Combining the above displays yields
		\begin{equation*}
			\| \nabla_y \nabla_x^l \bchi_{q,m,k}(x,\cdot) \|_{L^2(\TT^d)}
			\leq
			\frac{(\lambda+|x|^2)^{\frac32(q-l)}}{\gamma(\lambda)^{(q-m-2)_+}}
			\exp\bigl( A^{q-m+k} \bigr)
			\,.
		\end{equation*}
		This completes the induction step, and the proof of the proposition. 
	\end{proof}

	\subsection{Higher Order Expansions for Simple Eigenvalues} \label{ss.simple.high}
	Given the explicit construction of higher order correctors $\{\bchi_{q,m,k}\}$, along with their homogenized tensors $\bba_{q,m,k}$, the sequence $\{\mu_k\}_k,$ and smooth functions $\{U_k\}_k,$ we are now ready to prove Theorem \ref{t.simple.full} on the higher order expansion of a simple eigenvalue. 
	
	\begin{proof}[Proof of Theorem \ref{t.simple.full}]
		The proof of this theorem proceeds similarly to that of Theorem \ref{t.simple}, and we follow similar steps-- naturally, the associated computations are more involved. 
		
		We let $P \in \NN$ be an integer that will be fixed at the end of the proof. 
		
		\smallskip
		\emph{Step 1. } We set 
		\begin{equation*}
			\widetilde{\lambda}_\e 
			:=
			\lambda_0 + \e \mu_1 + \ldots + \e^P \mu_P,
		\end{equation*}
		and 
		\begin{equation*}
			w_\e(x) := \sum_{p=0}^P \sum_{k=0}^p \sum_{m=0}^{p-k} \nabla^m U_k(x) : \bchi_{p-k,m,k}\bigl(x,\frac{x}{\e} \bigr)\,. 
		\end{equation*}
		Then, the derivaton leading up to \eqref{e.macroscopic.derivation} shows that 
		\begin{align*}
			&-\nabla \cdot \bha^\e \nabla w_\e + (W(x) - \widetilde{\lambda}_\e) w_\e \\
			&\quad = \sum_{p=0}^P \e^p \biggl( (\mathcal{L}_0 - \lambda_0) U_p - \sum_{k=1}^{p-1}   \biggl( \nabla \cdot \sum_{m=1}^{p+1-k} \bba_{p+1-k,m,k} : \nabla^m U_k + \mu_{p-k} U_k\biggr)\biggr) + \nabla \cdot R_\e + S_\e,
		\end{align*}
		where, by Proposition \ref{p.exp.bounds} the functions $R_\e$ and $S_\e$ satisfy 
		\begin{equation} \label{e.Reps.high}
			\|R_\e\|_{L^2(\RR^d)} + \|S_\e\|_{L^2(\RR^d)} \leqslant C \e^P \frac{\lambda_0^{\sfrac{3P}2}}{\gamma({\lambda_0})^{P-1}} \exp(A^{P+1})\,. 
		\end{equation}

		\smallskip
		\emph{Step 2.} We set $\delta(\e,\lambda_0) := \frac{\e \lambda_0^{\sfrac32}}{\gamma(\lambda_0)} ,$ which, by \eqref{e.eps.condition} is smaller than $1.$ To complete the argument it remains to minimize the function $f(P) := \delta^P \exp(A^{P+1})\, $  over $P \in (1,\infty).$ Toward this goal, it is easily seen that $f(P) \to \infty$ as $P \to \infty,$ and $f(1) = \delta \exp(A^2) = O(\delta).$ At an interior critical point we must have $f^\prime(P) = 0, $ so that 
		\begin{equation*}
			0= \frac{f^\prime(P)}{f(P)} = \log \delta + A^{P+1}\log A\, ,
		\end{equation*}
		so that the optimal choice $P_*$, namely
		\begin{equation*}
			P_* \sim \frac{1}{\log A}\log \frac{|\log \delta|}{\log A}\, ,
		\end{equation*}
		and correspondingly,
		\begin{align*}
			f(P_*) &= \delta^P \exp(A^{P+1}) = \exp(P \log \delta + A^{P+1}) \\
			&= \exp \biggl( \frac{\log \delta}{\log A} \log \frac{|\log \delta|}{\log A} + \frac{|\log \delta| }{\log A} \biggr)\\
			&= \exp \bigl( - |\log_A\delta|\log |\log_A\delta| + |\log_A\delta| \bigr)\,. 
		\end{align*}
		It follows that $f(P_*) \leqslant \rho(\delta(\e,\lambda_0)) = \rho \bigl( \frac{\e \lambda_0^{\sfrac32}}{\gamma(\lambda_0)}\bigr)\,,$ where $\rho(t) := t^{c\log |\log t|}.$ Inserting this in \eqref{e.Reps.high} we obtain that 
		\begin{equation*}
			\|R_\e\|_{L^2(\RR^d)} + \|S_\e\|_{L^2(\RR^d)}  
			\leqslant 
			\gamma(\lambda_0)
			\rho\biggl(\frac{\e \lambda_0^{\sfrac32}}{\gamma(\lambda_0)}\biggr)\,. 
		\end{equation*}
		The proof is now completed exactly as in Step 3 of Theorem \ref{t.simple}.  
	\end{proof}
	
	\section{Expansions for High Multiplicity Eigenvalues and their Eigenfunctions} \label{s.multiple}

	\subsection{First Order Expansions for Multiple Eigenvalues} \label{ss.multiple.first}
	In this section we consider expansions for eigenvalues that are of high multiplicity. To be precise, let $\lambda_{0,j} = \lambda_{0,j+1} = \ldots = \lambda_{0,j+N-1} $ be an eigenvalue of $\mathcal{L}_0$ of multiplicity $N> 1,$ and let $\{\phi_{0,j+r}\}_{r=0}^{N-1}$ denote the associated eigenfunctions of $\mathcal{L}_0.$ Here, as usual we have used the enumeration of the eigenvalues of $\mathcal{L}_0$ in nondecreasing order, repeated according to multiplicity. 
	\smallskip
	We seek to expand the eigenvalues $\{\lambda_{\e,j+r}\}_{r=0}^{N-1}$ of the operator $\mathcal{L}_\e,$ and their associated eigenfunctions. Toward this goal we begin with a  preliminary lemma that we will crucially use. 
	
	\smallskip
	
	Next we define the matrix $\mathbb{D}$ via 
	\begin{equation}
		\label{e.Drs.def}
		\mathbb{D}_{rs} :=  \sum_{i=1}^d \int_{\RR^d} \bba_{3,e_i,0} \cdot \nabla \phi_{0,j+r} \partial_{x_i} \phi_{0,j+s} \,dx + \sum_{|\alpha|=2} \int_{\RR^d} \bba_{3,\alpha,0} \partial_x^\alpha \phi_{0,j+s} \cdot \nabla \phi_{0,j+r}\,dx  \, .
	\end{equation}
	The next lemma collects properties of $\mathbb{D}$ that will be crucially used in the sequel. 
	\begin{lemma}
		\label{l.Dsym}
		The matrix $\mathbb{D}$ satisfies
		\begin{equation}
			\label{e.good.Drs}
			\mathbb{D}_{rs}
			=
			\sum_{i,k=1}^d 
			\bigl \langle \bchi_{1,e_k,0} \bchi_{1,e_i,0} \bigr\rangle
			\int_{\RR^d} 
			(W(x) - \mu_0) 
			\partial_{x_k}\phi_{0,j+r}(x) \partial_{x_i}\phi_{0,j+s}(x)
			\,dx\,.
		\end{equation}
		In particular,~$\mathbb{D}$ is symmetric. 
	\end{lemma}
	\begin{proof}
		Repeating the proof of Lemma \ref{l.3rdorder.symm}, i.e., utilizing that when $|\alpha| = 2,$ then $\bba_{3,\alpha,0}$ is constant, and integrating by parts three times yields that the second group of terms in $\mathbb{D}_{rs}$ evaluate to zero. 
		It therefore  remains to compute the first term. 
		Using the definition of the higher order homogenized tensors, we have 
		\begin{align*}
			\bba_{3,e_i,0} = \langle \bha \nabla_y \bchi_{3,e_i,0} \rangle \,,
		\end{align*}
		where the higher order corrector $\bchi_{3,e_i,0}$ is the unique mean-zero (in $y$) solution to 
		\begin{align*}
			\nabla_y \cdot \bha \nabla \bchi_{3,e_i,0} = (W(x) - \mu_0) \bchi_{1,e_i,0}\,.
		\end{align*}
		Testing this equation (in the fast variable) with $\bchi_{1,e_k,0}$ and using the PDE satisfied by the first order corrector $\bchi_{1,e_k,0}$ yields 
		\begin{align*}
			- (\bha_{3,e_i,0})_k = - e_k \cdot \int_{\TT^d} \bha \nabla_y \bchi_{3,e_i,0}\,dy &= 	\int_{\TT^d} \nabla \bchi_{1,e_k,0} \cdot \bha \nabla \bchi_{3,e_i,0} \,dy\\
			& = - (W(x) - \mu_0) \int_{\TT^d} \bchi_{1,e_k,0} \bchi_{1,e_i,0}\,d y\,.
		\end{align*}
		This completes the proof. 
	\end{proof}

	We are now ready to prove Theorem \ref{t.multiple}. 
	\begin{proof}[Proof of Theorem \ref{t.multiple}.] We have shown in Lemma \ref{l.Dsym} that $\mathbb{D}$ is symmetric. Now we let $r \in \{0,\ldots, N-1\},$ and let $\mu_{2,j+r}$ denote the $r$-th eigenvalue of the matrix $\mathbb{D}$ along with the eigenvector $\bbe^r$ (by symmetry and our assumption, we can arrange the $\mu_{2,j+r}$ in increasing order). We also define 
		\begin{align*}
			U_{0,j+r} := \sum_{s=0}^{N-1} e^r_s \phi_{0,j+s}\,. 
		\end{align*}
		Here, $\mathbf{e}^r := (e^r_1,\ldots, e^r_n)$ denotes the eigenvector of the matrix $\mathbb{D}$ associated to the eigenvalue $\mu_{2,j+r}.$ 
		Then, clearly, $\mathcal{L}_0 U_{0,j+r} = \lambda_{0,j} U_{0,j+r},$ and 
		\begin{align} \label{e.orthog}
			\int_{\RR^d} U_{0,j+r} U_{0,j+s} = \left\{ 
			\begin{array}{cc}
				1 & r = s \\
				0 &  r \neq s.
			\end{array}
			\right.
		\end{align}
		As in the proof of Theorem \ref{t.simple}, we let 
		\begin{equation*}
			\widetilde{\lambda}_{\e,j+r} := \lambda_{0,j} + \e^2 \mu_{2,j+r}\,,
		\end{equation*}
		and we set 
		\begin{equation*}
			w_{\e,j+r} := U_{0,j+r} + \e^2 U_{2,j+r} + \e \nabla (U_{0,j+r} + \e^2 U_{2,j+r}) \cdot \bchi^{(1)}(\tfrac{x}{\e})\,, 
		\end{equation*}
		where $U_{2,j+r}$ is the unique solution to 
		\begin{equation*}
			(\mathcal{L}_0 - \mu_0)U_{2,j+r} = \mu_{2,j+r} U_{0,j+r} + \nabla_x \cdot \sum_{m=1}^2 \sum_{|\alpha| = m} \bba_{3,\alpha,0} \partial_x^\alpha U_{0,j+r}\,,
		\end{equation*}
		which is orthogonal to each of $\{\phi_{0,j+s}\}_{s=0}^{N-1}.$ 
		
		We observe that by choice of $\mu_{2,j+r},$ such a solution exists for each $r = 0,\ldots, N-1.$ By linearity, this solution is also orthogonal to $\{U_{0,j+r}:r=0,\ldots, N-1\}.$ 
		
		Finally, as in the proof of Theorem \ref{t.simple}, for fixed $x \in \RR^d$ we let $z_{j+r}(x,\cdot): \TT^d \to \RR^d$ denote the unique $H^1(\TT^d)$ mean-zero solution to the PDE 
		\begin{equation*}
			\nabla \cdot \bha \nabla z_{j+r}(x,y) = (W(x) - \lambda_{0,j}) \nabla U_{0,j+r} \cdot \bchi^{(1)}(y)\,. 
		\end{equation*}
		We then set $z_{\e,j+r}(x) := \e^3 z_{j+r}(x,\sfrac{x}{\e})\,.$ 
		
		\smallskip
		\emph{Step 2. } In this step we compute $	(\mathcal{L}_\e - \widetilde{\lambda}_\e) w_\e.$ Proceeding as in the proof of Theorem \ref{t.simple} we find 
		\begin{align*}
			(\mathcal{L}_\e - \widetilde{\lambda}_\e) (w_\e - z_{\e,j+r}) \, &= (\mathcal{L}_0 - \la_0) U_{0,j+r} \\
			&+ \e^2 \biggl( (\mathcal{L}_0 - \la_0) U_2 + (W(x) - \la_0) U_2 - \mu_2 U_0 \\ &\qquad \qquad - \nabla_x \cdot \sum_{m=1}^2 \sum_{|\alpha| = m} \bba_{3,\alpha,0} \partial_x^\alpha U_0 \biggr) \\
			&+ \nabla \cdot \biggl( \sum_{k=1}^d s^\e_{e_k} - \chi^{1,\e}_{e_k} \bha^\e\biggr)\nabla \partial_{x_k} (U_0 + \e^2 U_2)\\
			&+ \e^2 \nabla_x \cdot \sum_{m=1}^2 \sum_{|\alpha|=m} \bba_{3,\alpha,0} \partial_x^\alpha U_0\\
			&+ \e^3 (W(x) - \la_0) \nabla U_2 \cdot \chi^{(1)}(\tfrac{x}{\e})   + (W(x) - \la_0) z_{\e,j+r}\\
			&-\e^3 \mu_2 \nabla (U_0 + \e^2 U_2) \cdot \chi^{(1)}(\tfrac{x}{\e}) - \e^4 \mu_2 U_2\,.
		\end{align*}
		
		By definition, the first two lines of the preceding display are zero. By the computations in Theorem \ref{t.simple} (specifically, those involving the second-order corrector equation and leading up to \eqref{e.wepsdivgrad}), and using the symmetry property of $\bba^{(3)}$ from lemma \ref{l.3rdorder.symm}, we obtain that the third line rewrites in divergence form as $\nabla \cdot R_\e,$ with $R_\e$ satisfying the bound 
		\begin{equation*}
			\|\nabla \cdot R_\e\|_{H^{-1}(\RR^d)} \leqslant \|R_\e\|_{L^2(\RR^d)} \leqslant C\e^2 \la_0^{\sfrac32} + C\e^3 \frac{\la_0^{\sfrac52}}{\gamma(\la_0)}\,.
		\end{equation*} 
		
		Writing the term in the fourth line, which is in divergence form, as $\nabla \cdot \tilde{R}_\e$ , we easily have the estimate 
		\begin{equation*}
			\|\nabla \cdot \tilde{R}_\e\|_{H^{-1}(\RR^d)} \leqslant C \e^2 \la_0\,. 
		\end{equation*}
		Finally, combining the fifth and sixth lines and denoting them by $S_\e,$ we estimate similarly to \eqref{e.seps3bnd}
		\begin{equation*}
			\|S_\e\|_{L^2(\RR^d)} \leqslant C\e^3 \la_{0,j}^{\sfrac52} + C\frac{\e^2\la_{0,j}^2}{\gamma(\la_{0,j})} \,. 
		\end{equation*}
		Therefore we can write 
		\begin{equation} \label{e.PDEw-z}
			(\mathcal{L}_\e - \widetilde{\lambda}_{\e,j+r}) (w_{\e,j+r} - z_{\e,j+r}) = \nabla \cdot \overline{R}_{\e,j+r} + \tilde{S}_{\e,j+r},
		\end{equation}
		with 
		\begin{equation*}
			\|\nabla \cdot \overline{R}_{\e,j+r}\|_{H^{-1}(\RR^d)} + \|\tilde{S}_{\e,j+r}\|_{L^2(\RR^d)} \leqslant C \e^2 \lambda_{0,j}^{\sfrac32} + C\frac{\e^2 \lambda_{0,j}^3}{\gamma(\lambda_{0,j})} \,. 
		\end{equation*}
		In order to proceed as in Step 2 of the proof of Theorem \ref{t.simple}, we make some preliminary observations about a convenient basis in which to solve \eqref{e.PDEw-z}. We note that for any $r,s \in \{0,\ldots, N-1\},$ in light of \eqref{e.orthog}, 
		\begin{align*}
			&\int_{\RR^d} w_{\e,j+r}w_{\e,j+s} \,dx \\ &\quad = \int_{\RR^d} \bigl(U_{0,j+r} + \e U_{2,j+r} + \e \nabla (U_{0,j+r} + \e U_{2,j+r}) \cdot \bchi^{(1)}(\tfrac{x}{\e}) \bigr)\\ &\quad \quad \quad \quad \times \bigl(U_{0,j+s} + \e U_{2,j+s} + \e \nabla (U_{0,j+s} + \e U_{2,j+s}) \cdot \bchi^{(1)}(\tfrac{x}{\e}) \bigr)\,dx \\
			&\quad =\int_{\RR^d} U_{0,j+r}U_{0,j+s}\,dx + \e \biggl( \int_{\RR^d}U_{2,j+r}U_{0,j+s}  + \int_{\RR^d} U_{0,j+r}U_{2,j+s}  \biggr) \\ &\quad \quad +  \e \biggl( \int_{\RR^d} U_{0,j+r}\nabla U_{0,j+s} \cdot \bchi^{(1)}(\tfrac{x}{\e}) + U_{0,j+s} \nabla U_{0,j+r}\cdot \bchi^{(1)}(\tfrac{x}{\e})  \biggr) +  \delta_\e(\lambda_{0,j}),
		\end{align*}
		with 
		\begin{equation*}
			|\delta_\e(\lambda_{0,j})| \leqslant C\e^2 \lambda_{0,j}^{\sfrac32} + C \frac{\e^2 \lambda_{0,j}^3}{\gamma(\lambda_{0,j})}\,.
		\end{equation*}
		Now, by construction, $\int_{\RR^d} U_{2,j+r} U_{0,j+s} = 0,$ and $\int_{\RR^d} U_{0,j+r} U_{2,j+s} = 0,$ and for the other $O(\e)$ term in the computation above, we observe that 
		\begin{align*}
			&\biggl|\int_{\RR^d} U_{0,j+r}\nabla U_{0,j+s} \cdot \bchi^{(1)}(\tfrac{x}{\e}) + U_{0,j+s} \nabla U_{0,j+r}\cdot \bchi^{(1)}(\tfrac{x}{\e}) \biggr|\\
			&\quad \leqslant \e \|U_{0,j+r}\nabla U_{0,j+s}\|_{H^1(\RR^d)}\|\bchi^{(1)}\|_{H^{-1}(\RR^d)} \leqslant \delta_\e(\lambda_{0,j})\,. 
		\end{align*}
		In light of \eqref{e.orthog} it then follows then that the set $\{w_{\e,j+r}\}_{r=0}^{N-1}$ is approximately orthogonal in $L^2(\RR^d):$
		\begin{equation} \label{e.almost.orthog}
			\int_{\RR^d} w_{\e,j+r}w_{\e,j+s} = \delta_{rs} + \delta_\e(\lambda_{0,j})^2\,. 
		\end{equation}
		In order to complete the argument, we let $\{\psi_{\e,j+r}\}_{r=0}^{N-1}$ denote the eigenvalues of $\mathcal{L}_\e$ associated with $\lambda_{\e,j},\ldots, \lambda_{\e,j+N-1},$ respectively, that are normalized according to the conditions 
		\begin{equation*}
			\int_{\RR^d}\psi_{\e,j+s}\phi_{0,j+r}\,dx = e^r_s\,.
		\end{equation*}
		For each $r,s \in \{0,\ldots, N-1\},$ we set
		\begin{equation*}
			\begin{aligned}
				d_{\e,j+r,j+s} &:= \int_{\RR^d} \bigl( w_{\e,j+r} - z_{\e,j+r}\bigr) \frac{\psi_{\e,j+s}}{\|\psi_{\e,j+s}\|_{L^2(\RR^d)}}\,dx\,.
			\end{aligned}
		\end{equation*}
		Then by the triangle inequality and \eqref{e.almost.orthog} we find
		\begin{equation*}
			\begin{aligned}
				&|d_{\e,j+r,j+s} - \delta_{rs}| 
				\leqslant \biggl|\int_{\RR^d} (w_{\e,j+r}- z_{\e,j+r}) \biggl(\frac{\psi_{\e,j+s}}{\|\psi_{\e,j+s}\|_{L^2(\RR^d)}} - (w_{\e,j+s} - z_{\e,j+s})\biggr)\biggr| 
				\\&\quad
				\leqslant \biggl\|\frac{\psi_{\e,j+s}}{\|\psi_{\e,j+s}\|_{L^2(\RR^d)}} - (w_{\e,j+s}- z_{\e,j+r})\biggr\|_{L^2(\RR^d)} + \delta_\e(\lambda_0)\,. 
			\end{aligned}
		\end{equation*}
		From \eqref{e.PDEw-z} we find 
		\begin{equation*}
			\sum_{s=0}^{N-1} (\lambda_{\e,j+s} - \widetilde{\lambda}_{\e,j+r})^2 d_{\e,j+r,j+s}^2 \leqslant \delta_\e(\lambda_{0,j})^2\,. 
		\end{equation*}
		Combining the last two displays yields 
		\begin{equation*}
			|\lambda_{\e,j+r} - \widetilde{\lambda}_{\e,j+r}| \leqslant \delta_\e(\lambda_{0,j}) \Biggl( 1 + \biggl\|\frac{\psi_{\e,j+r}}{\| \psi_{\e,j+r}\|_{L^2(\RR^d)}} - (w_{\e,j+r} - z_{\e,j+r})\biggr\|_{L^2(\RR^d)}\Biggr)\,.
		\end{equation*}
		To complete the argument, we must estimate the convergence rates for the eigenfunctions, and for this let us note that 
		\begin{align*}
			\lefteqn{
				(\mathcal{L}_\e - \widetilde{\lambda}_{\e,j+r})\biggl(\frac{\psi_{\e,j+r}}{\| \psi_{\e,j+r}\|_{L^2(\RR^d)}} - (w_{\e,j+r} - z_{\e,j+r})\biggr) 
			} \qquad & 
			\notag \\ & 
			= (\lambda_{\e,j+r} - \widetilde{\lambda}_{\e,j+r}) \frac{\psi_{\e,j+r}}{\| \psi_{\e,j+r}\|_{L^2(\RR^d)}} + \nabla \cdot \overline{R}_{\e,j+r} + \tilde{S}_{\e,j+r}\,.
		\end{align*}
		so that, 
		\begin{align*}
			&\biggl\|\frac{\psi_{\e,j+r}}{\| \psi_{\e,j+r}\|_{L^2(\RR^d)}} - (w_{\e,j+r} - z_{\e,j+r})\biggr\|_{H^1(\RR^d)} \leqslant |\lambda_{\e,j+r} - \widetilde{\lambda}_{\e,j+r} | + \delta_{\e}(\lambda_{0,j})\\
			& \quad \leqslant \delta_\e(\lambda_{0,j}) \biggl( 1 + \biggl\|\frac{\psi_{\e,j+r}}{\| \psi_{\e,j+r}\|_{L^2(\RR^d)}} - (w_{\e,j+r} - z_{\e,j+r})\biggr\|_{L^2(\RR^d)}\biggr) + \delta_{\e}(\lambda_{0,j})\,.
		\end{align*}
		The proof is finished by buckling, and using the triangle inequality. 
	\end{proof}
	
	\subsection{Rigorous Construction and Estimates}
	\label{ss.rig.multiple}
	As is well-known from perturbation theory, if $\la_{0,j}$ is an eigenvalue of $\mathcal{L}_0$ with multiplicity $N > 1$ then we must construct all $N$ branches of eigenpairs splitting off of $\la_{0,j}$ together, since the branches interact with each other. As in the simple case, our construction of the higher order branches will be inductive. 
	
	\smallskip
	\emph{Base case.} 
	\begin{itemize}
		\item For each $s := \{0,\ldots, N-1\},$ we set $\mu_0 := \lambda_{0,j},$ and let $\{\phi_{0,j+r}\}_{r=0}^{N-1}$ denote the $N$ orthonormal eigenfunctions of $\mathcal{L}_0$ associated to the eigenvalue $\lambda_{0,j},$ which is the $j$th eigenvalue of $\mathcal{L}_0$ in an enumeration of the eigenvalues in nondecreasing order. In order to ease notation, we will largely suppress the dependence on the index $j $ in the remainder of this section. 
		\item For all $N$ branches, we initialize our construction by setting $\bchi_{0,0,k} = 1,$ and we define $\bbf_{0,0,k} := 0$ and $\bba_{0,0,k} := 0$ for every $k \in \NN_0.$ 
		\item $\bchi_{q,0,k} := 0$, as well as $\bbf_{q,0,k} := 0,$ and $\bba_{q,0,k} := 0,$ for every $q \in \NN,k \in \NN_0.$
		\item $\bchi_{q,m,k} := 0,$ as well as $\bbf_{q,m,k} := 0$ and $\bba_{q,m,k} := 0$ for every $q,m,k \in \NN_0$ with $m > q.$
		\item We define $\bchi_{1,1,k},$ for each $k \in \NN_0$ to be the usual first-order corrector in classical periodic homogenization, that is, the unique solution of \eqref{e.hocorr11}. Also we define 
		\begin{align*}
			\bbf_{1,1,k} := \bha(I + \nabla_y \bchi_{1,1,k}), \quad \mbox{ and } \bba_{1,1,k} := \langle \bbf_{1,1,k} \rangle\,.
		\end{align*}
		As usual, $\bba_{1,1,k} = \bba$ is the usual homogenized matrix of classical periodic homogenization. Note that $\bchi_{1,1,k}, \bbf_{1,1,k}$ and $\bba_{1,1,k}$ depend neither on the index $k$, nor on the slow variable $x.$
		\item We define $\bchi_{2,1,k} := 0,$ as well as $\bbf_{2,1,k} := 0$ and $\bba_{2,1,k} := 0$, for every $k \in \NN_0.$
		\item We define $\bchi_{2,2,k}$ for each $k \in \NN_0$, to be the unique mean-zero solution of \eqref{e.hocorr1}, define $\bbf_{2,2,k} := \bha \nabla_y \bchi_{2,2,k} + \bha \otimes \bchi_{1,1,k},$ and $\bba_{2,2,k} := \langle \bbf_{2,2,k} \rangle\,. $
		\item We \emph{define} that 
		\begin{align} \label{e.ansatzU0}
			U_{0,j+r} = \sum_{s=0}^{N-1} e^r_s \phi_{0,j+s} \, ,
		\end{align}
		where $\{e^r_s\}_{r,s=0}^{N-1}$ are the eigenfunctions of the symmetric matrix $\mathbb{D}$, which form an orthonormal basis of $\RR^N$.  We will denote the $N$ eigenvalues of $\mathbb{D}$ by $\{\mu_{2,j+r}\}_{r=0}^{N-1}$ (we recall $\lambda_{0,j}= \ldots = \lambda_{0,j+N-1}$ and $j$ denotes the lowest index such that $\lambda_{0,j}$ is an eigenvalue of $\mathcal{L}_0.$ )
		\item We set  $\mu_{1,j+r} = 0$ and also set $U_{1,j+r}\equiv 0$  for every $r \in \{0,\ldots, N-1\}.$ 
		
	\end{itemize}
	\emph{Induction Step.} 
	Let us suppose that, for some integer $K \in \NN, K \geqslant 2,$ we have defined $\bchi_{q,m,k,j+s},$ the associated fluxes $\bbf_{q,m,k,j+s}$ and the homogenized coefficients $\bba_{q,m,k,j+s}$ for indices 
	\begin{align}
		\label{e.qmk.induction.mult}
		(q,m,k) \in J(K) := \{ (q,m,k) : m \in \NN_0, 0 \leqslant k \leqslant K, 0 \leqslant q \leqslant K+2-k, 0 \leqslant s \leqslant N-1\}\,,
	\end{align}
	The higher order correctors depend on the specific branch $s,$ and this is the reason for the last index in each of these objects, next to the three familiar ones from the simple eigenvalue case. 
	
	Additionally in the induction hypothesis, we assume that $\mu_{k,j+s}$ have been defined for every $k \in \{0,\ldots, K-2\}, s \in \{0,\ldots, N-1\}$ along with macroscopic functions $U_{k,j+s},$ and these functions satisfy the normalization conditions 
	\begin{align*}
		\int_{\RR^d} U_{k,j+s} \phi_{0,j+r} \,dx = \alpha_{k,s,r}, \quad \quad k \in \{1,\ldots, K-4\}, \,r,s \in \{0, \ldots, N-1\}\,.
	\end{align*}
	
	In the induction step we give ourselves the task of determining the $N$ branches of macroscopic correctors at order $\{U_{K-1,j+r}\}_{r=0}^{N-1}$, and $\{\mu_{K-1,j+r}\}_{r=0}^{N-1}$, \emph{as well as the normalization conditions} $\alpha_{K-3,s,t}$ for the function $\{U_{K-3,s}\}_{s=0}^{N-1}.$ We point out that, as we will explain below, the normalization conditions for a given stage arises two stages further in the inductive construction as part of a solvability criterion. 
	
	Precisely, by the formal derivation of the homogenized equations, we recall that $U_{K-1,j+r} $ must satisfy 
	\begin{align}
		\label{e.UK-1def.mult}
		\lefteqn{
			(\mathcal{L}_0 - \lambda_0)U_{K-1,j+r} 
		} \qquad & 
		\notag \\ & 
		= 
		\nabla_x \cdot \biggl( \sum_{k=0}^{K-2} \sum_{m=0}^{K-1-k} \sum_{|\alpha| = m} \bba_{K-k,\alpha,k,j+r} \partial_x^\alpha U_{k,j+r}\biggr) + \sum_{k=0}^{K-2} \mu_{K-1-k,j+r} U_{k,j+r}\,.
	\end{align}
	Solvability for this PDE requires that the right-hand side of \eqref{e.UK-1def.mult} be orthogonal to each of $\{\phi_{0,j+t}\}_{t=0}^{N-1}.$ Imposing this, and using \eqref{e.ansatzU0} yields 
	\begin{equation*}
		\begin{aligned}
			&	\int_{\RR^d} \sum_{k=0}^{K-3} \sum_{m=0}^{K-1-k} \sum_{|\alpha| = m} \bba_{K-k,\alpha,k,j+s} \cdot \nabla \phi_{0,j+t} \partial_x^\alpha U_{k,j+s}\,dx\\
			& = \sum_{k=0}^{K-1}  \mu_{K-1-k,j+s} \int_{\RR^d} U_{k,j+s}\phi_{0,t}\,dx + \mu_{K-1,j+s} \int_{\RR^d} U_{0,s}\phi_{0,t} \,dx + \mu_{2,s} \int_{\RR^d} U_{K-3,j+s}\phi_{0,t}\,dx\,. 
		\end{aligned}
	\end{equation*}
	Observe carefully that the terms corresponding to the index $k=K-2$ do not appear in the preceding display: this is because $\mu_{1,j+s} = 0$ for each $s,$ and the homogenized coefficients $\bba_{2,\alpha,k, \cdot} $ vanishes when $|\alpha| = 1,$ by the base-case. Toward determining the normalizations $\alpha_{K-3,s,t},$ let us write 
	\begin{align*}
		U_{K-3,j+s} := \mathring{U}_{K-3,j+s} + \sum_{t=0}^{N-1} \alpha_{K-3,s,t}\phi_{0,t}\,, \quad s = 0,\ldots, N-1\,.
	\end{align*} 
	By the induction hypothesis, $\mathring{U}_{K-3,j+s}$ , which is the unique solution to the PDE for $U_{K-3,j+s}$ which is orthogonal to $\{\phi_{0,t}\}_{t=0}^{N-1}$ exists. Inserting this decomposition in the preceding display yields the following problem for $\mu_{K-1,s}$ and $\{\alpha_{K-3,s,t}\}_{t=0}^{N-1}\,:$
	\begin{equation} \label{e.D-mu}
		(\mathbb{D} - \mu_{2,j+s}) \begin{pmatrix}
			\alpha_{K-3,s,0}\\
			\vdots\\
			\alpha_{K-3,s,N-1}
		\end{pmatrix} = \mu_{K-1,j+s} \begin{pmatrix}
			e_{s,0}\\
			\vdots\\
			e_{s,N-1}
		\end{pmatrix} + \mathbb{F}_{j+s}\,,
	\end{equation} 
	here $\mathbb{F}_{j+s} \in \RR^N$ is defined via 
	\begin{align*}
		(	\mathbf{F}_{j+s} )_t &:= - \sum_{k=1}^{K-4} \sum_{m=0}^{K-1-k}\sum_{|\alpha| = m} \int_{\RR^d} (\bba_{K-k,\alpha,k,j+s} \cdot \nabla \phi_{0,t}) \partial_x^\alpha U_{k,j+s}\,dx \\ &\quad + \sum_{k=1}^{K-4} \mu_{K-1-k,j+s} \int_{\RR^d} U_{k,j+s}\phi_{0,j+t}\,dx\,. 
	\end{align*}
	
	At this point we use our assumption that $\mu_{2,j+s}$ is a simple eigenvalue of $\mathbb{D}$ for each $s \in \{0,\ldots, N-1\}$ and that the associated eigenvector is $\mathbf{e}_s := (e_{s,0},\ldots, e_{s,N-1})^t$. Taking the inner product of \eqref{e.D-mu} with the unit vector $\mathbf{e}_s$ and using the symmetry of $\mathbb{D}$ from Lemma \ref{l.Dsym}, we find 
	\begin{align*}
		\mu_{K-1,j+s} = - \mathbb{F}_{j+s} \cdot \mathbf{e}_s\,. 
	\end{align*}
	
	By elementary linear algebra then, since the right-hand side of \eqref{e.D-mu} is orthogonal to the kernel of $\mathbb{D} - \mu_{2,j+s},$ it follows that we can invert \eqref{e.D-mu} to find the undetermined coefficients $(\alpha_{K-3,s,t})_{s,t=0}^{N-1},$ and in turn, since $U_{K-3}$ is then uniquely determined, the existence of $U_{
		K-1,j+r}$ is obtained for each $r = 0,\ldots, N-1.$ The macroscopic function $U_{K-1}$ is nonunique up to an arbitrary function in the kernel of $\mathcal{L}_0 - \la_0,$ which, as in the induction step, is determined as part of the solvability condition for $U_{K+1}.$ Towards completing the inductive construction, we notice that 
	
	\begin{itemize}
		\item The above argument uniquely determines, for each $s \in \{0,\ldots, N-1\},$ $\mu_{K-1,j+s} \in \RR.$ It also determines $U_{K-1,j+s}\in L^2(\RR^d)$ which is unique upto addition of linear combinations of $\{\phi_{0,j+r}\}_{r=0}^{N-1}.$ 
		\item Fixing $s = 0,\ldots, N-1,$ the functions $\bchi_{2,m,k}$, as well as $\bbf_{2,m,k}$ and $\bba_{2,m,k},$ have already been defined in the base case above, for every $k$, and in particular for $k = K+1.$ The only nonzero object among these is $\bchi_{2,2,k},$ (and therefore only $\bba_{2,2,k}).$ Notice these objects do not depend on $s,$ nor on the macroscopic variable $x.$  
		\item We define, $\bchi_{q,m,k,j+s}$ for each $(q,m,k) \in J(K+1)\setminus J(K),q\neq 2, s \in \{0,\ldots, N-1\}$, to be the unique solution of 
		\begin{equation} \label{e.hocorrq.m} \left\{
			\begin{aligned}
				-\nabla_y \cdot \bha \nabla_y \bchi_{q,m,k,j+s} & = \nabla_y \cdot (\bha \otimes \bchi_{q-1,m-1,k,j+s})  + \nabla_y \cdot \bha\nabla_x \bchi_{q-1,m,k,j+s}  \\ & \qquad + \mathring{\bbf}_{q-1,m-1,k,j+s}   + \nabla_x \cdot \mathring{\bbf}_{q-1,m,k,j+s}   - W(x) \mathring{\bchi}_{q-2,m,k,j+s}  \\ & \qquad
				+ \sum_{r=m+k}^{q-2+k} \mu_{q-2+k-r,j+s} \mathring{\bchi}_{r-k,m,k,j+s} \quad \quad \mbox{ in } \mathbb{T}^d,\\
				\langle \bchi_{q,m,k,j+s} \rangle = 0.
			\end{aligned} \right.
		\end{equation}
		Note that~$q\neq 2$ implies~$q\geq 3$ and thus~$k\leq K$. Therefore all the terms on the right side have been defined already, by the induction hypothesis, because all the index triples belong to~$J(K)$
		and the highest index~$i$ of~$\mu_i$ that appears in \eqref{e.hocorrq} is~$i=K-1$, which we have above defined in~\eqref{e.muk-1def}. Moreover, the right side of the equation has zero mean and so the equation is uniquely solvable. 
		We then define 
		\begin{equation} \label{e.bbf.rig.m}
			\bbf_{q,m,k,j+s} := \bha \nabla_x \bchi_{q-1,m,k,j+s} + \bha \otimes \bchi_{q-1,m-1,k,j+s} + \bha \nabla_y \bchi^s_{q,m,k},
		\end{equation}
		and then
		\begin{equation} \label{e.aqmk.rig.m}
			\bba_{q,m,k}(x) := \langle \bbf_{q,m,k,j+s}(x,\cdot) \rangle.
		\end{equation}
		This completes the inductive construction of all the objects we set out to construct. 
	\end{itemize}
	\smallskip
	By easy modifications of the arguments in the proof of Proposition \ref{p.exp.bounds}, we can prove:
	\begin{proposition} \label{p.exp.bounds.mult}
		There exists $C(d,\theta)<\infty$
		such that, for every~$q,m,k\in\NN$ with~$m\leq q$, and for each $s \in \{0, \ldots, N-1\},$
		\begin{equation}
			\label{e.reg.indyhyp.m}
			\left\{
			\begin{aligned}
				& |\mu_{k,j+s}| 
				\leq 
				\frac{\lambda_{0,j}^{\frac{3k}2}}{{\gamma(\lambda_{0,j})^{k-1}}}
				\exp \bigl( C^{k+1} \bigr) \,,\\
				& \nnn U_{k,j+s} \nnn_{\lambda,C^k}
				\leq 
				\frac{\lambda_{0,j}^{\frac{3k}2}}{{\gamma(\lambda_{0,j})^{k}}}
				\exp \bigl( C^{k+1} \bigr) \,, \\
				& \frac1{(q+l)!}\sup_{x \in \RR^d} 
				\bigl( \lambda_{0,j} +|x|^2\bigr)^{-\frac 12(q-l)}
				\| \nabla_x^l \bchi_{q,m,k,j+s}(x,\cdot) \|_{{H^1}(\mathbb{T}^d)} 
				\leq 
				\frac{\exp ( C^{q-m+k} )
				}{\gamma(\lambda_{0,j})^{(q-2-m)_+}}\,,
			\end{aligned}
			\right.
		\end{equation}
		where, we recall that $\gamma(\lambda_{0,j})$ is the spectral gap of the multiple eigenvalue $\lambda_{0,j}. $ 
	\end{proposition}

	\subsection{Higher Order Expansions for Multiple Eigenvalues} \label{ss.multiple.high}
	We are finally in a position to prove Theorem \ref{t.multiple.full}. 
	\begin{proof}[Proof of Theorem \ref{t.multiple.full}]
		The proof proceeds exactly like that of Theorem \ref{t.simple.full}, and is concluded like in the proof of Theorem \ref{t.multiple}. 
	\end{proof}

	\subsection*{Acknowledgments}
	S.A.~was partially supported by NSF grant DMS-2000200. 
	R.V.~was partially supported by the Simons Foundation (Award \# 733694) and an AMS-Simons travel award.

	\bibliographystyle{acm}
	\bibliography{ref}
	
\end{document}